\documentclass[11pt]{article}

\usepackage[a4paper,margin=1in]{geometry}
\usepackage{amsmath,amsthm,amsfonts,amssymb}
\usepackage{bm,mathrsfs,upgreek}
\usepackage{graphicx,booktabs,tabularx,float}
\usepackage[authoryear,round]{natbib}
\usepackage{authblk}
\usepackage[hypertexnames=false,colorlinks=true,linkcolor=blue,citecolor=blue,urlcolor=blue]{hyperref}
\usepackage[nameinlink,noabbrev]{cleveref}
\allowdisplaybreaks

\newtheorem{theorem}{Theorem}[section]

\newtheorem{assumption}[theorem]{Assumption}
\newtheorem{lemma}[theorem]{Lemma}
\newtheorem{proposition}[theorem]{Proposition}

\theoremstyle{definition}

\newtheorem{remark}[theorem]{Remark}

\def\X{{\bf X}}
\def\x{{\bf x}}
\def\Y{{\bf Y}}

\def\Z{{\bf Z}}

\def\A{{\bf A}}

\def\B{{\bf B}}
\def\R{{\bf R}}
\def\K{{\bf K}}

\def\bI{{\bf I}}
\def\bS{{\bf S}}

\def\mR{\mathbb{R}}

\def\tr{\mbox{tr}}
\def\var{\mbox{var}}

\def\cov{\mbox{cov}}

\def\diag{\mbox{diag}}

\newcommand{\one}{\mathbf{1}}
\newcommand{\E}{\mathbb{E}}

\newcommand{\trans}{^\top}

\newcommand{\brho}{\mbox{\boldmath $\rho$}}
\newcommand{\tY}{\widetilde{Y}}
\newcommand{\tbY}{\widetilde{\Y}}

\newcommand{\tT}{\widetilde{T}}
\newcommand{\tM}{\widetilde{M}}
\newcommand{\tV}{\widetilde{V}}
\def\ind{\,{\buildrel d \over =}\,}
\begin{document}
\title{Phase transition of Schott's statistic for high-dimensional heavy-tailed data}
\author[1]{Hantao Chen\thanks{Email: htchen2000@sjtu.edu.cn}}
\author[2]{Guangming Pan\thanks{Email: gmpan@ntu.edu.sg}}
\author[1]{Cheng Wang\thanks{Corresponding author. Email: chengwang@sjtu.edu.cn}}
\affil[1]{School of Mathematical Sciences, MOE-LSC, Shanghai Jiao Tong University,
Shanghai 200240, China}
\affil[2]{School of Physical and Mathematical Sciences, Nanyang Technological University,
Singapore 637371, Singapore}
\date{}
\maketitle

\begin{abstract}
Consider Schott's statistic \citep{schott2005testing} defined as the squared Frobenius norm of the sample correlation matrix for data from $\alpha$-regularly varying populations. We investigate its asymptotic distribution in a general framework characterized by the data dimension $p$, sample size $n$, and regularly varying index $\alpha$. In particular, we identify a phase transition phenomenon in the asymptotic behavior. For light-tailed populations ($\alpha > 3$), we revisit the $\alpha$-free asymptotic distribution but relax the constraint on the ratio of $p/n$. For heavy-tailed populations ($\alpha < 3$), we derive a new asymptotic normal distribution whose variance explicitly depends on $\alpha$. We also propose a consistent estimator for the asymptotic variance such that the standardized Schott's test statistic remains applicable for unknown location parameters and all $\alpha > 0$.
\end{abstract}

\noindent\textbf{Keywords:} Central limit theorem, Heavy-tailed data, High-dimensional inference, Random matrix theory, Schott's statistic

\medskip
\noindent\textbf{Mathematics Subject Classification:} 62H10, 60F05

\section{Introduction}
Suppose that the observations $\Z_1,\cdots,\Z_n\in\mR^p$ are independent and identically distributed (i.i.d.) from a $p$-dimensional population $\Z=(Z_1,\cdots,Z_p)$. Denote the data matrix by 
\begin{align}
	\begin{pmatrix}
		\Z_1\trans\\
		\vdots\\
		\Z_n\trans
	\end{pmatrix}=\begin{pmatrix}
		Z_{11}&\cdots&Z_{1p}\\
		\vdots&\cdots&\vdots\\
		Z_{n1}&\cdots&Z_{np}
	\end{pmatrix}=\left(\X_1,\cdots,\X_p\right).\label{1-April}
\end{align}
Consider the hypotheses testing problem 
\begin{align*}
	H_0: \mbox{The covariance matrix}, \cov(\Z), \mbox{is diagonal}.
\end{align*}
Equivalently, the population correlation matrix is an identity matrix, i.e.
\begin{align*}
	H_0:\mbox{cor}(\Z)=\bI_p. 
\end{align*}
Define the sample covariance matrix 
\begin{align*}
	\bS=\frac{1}{n}\sum_{i=1}^n \Z_i \Z_i\trans,
\end{align*}
and the corresponding sample correlation matrix
\begin{align*}
	\R=&\diag(\bS)^{-\frac{1}{2}} \bS \diag(\bS)^{-\frac{1}{2}}\label{form:def_R},
\end{align*}
where
\begin{align*}
	\diag(\bS)=\frac{1}{n}\diag\left(\|\X_1\|^2,\cdots,\|\X_p\|^2 \right) .
\end{align*}
\cite{schott2005testing} proposed a statistic (Schott's statistic) for the above testing problem
\begin{align}
	T=\tr(\R^2).
\end{align}
Note that $\tr(\R)=p$. Schott's statistic eventually evaluates the squared Frobenius norm of the difference between the sample correlation matrix $\R$ and the identity matrix $\bI$, i.e.,
\begin{align*}
	T=\tr(\R^2)-2 \tr(\R)+2p=\|\R-\bI_p\|_F^2+p.
\end{align*}
\cite{schott2005testing} obtained the following result. 
\begin{theorem}[Theorem 1 of \citealt{schott2005testing}]\label{thm:Schott}
	Suppose $\Z \sim N(\bf{0}, \bI_p)$ and $p/n \to y \in(0,\infty)$. Then
	\begin{align}\label{clt0} 
		\frac{n}{2p}\left(\tr(\R^2)-p-\frac{p(p-1)}{n}\right)\stackrel{d}{\to}N(0,1).
	\end{align}
\end{theorem}

In random matrix theory, studying the limiting behavior of the linear spectral statistics (including $\tr(\R^2)$ and $\log|\R|$) of sample correlation matrices has been a central topic. Specifically, assuming $Z_1,\cdots, Z_p$ are i.i.d. random variables, \cite{gao2017high} show the asymptotic normality \eqref{clt0} still holds whenever $\E(Z_{ij}^4)<\infty$. See also \cite{mestre2017correlation} and \cite{zheng2019test}.

Since the sample correlation matrix is self-normalized, it is interesting to study the robustness of its linear spectral statistics  to heavy-tailed populations. \cite{heiny2024log} first observed a surprising fact that unlike the CLT for $\log|\bS|$, the CLT for $\log|\R|$ holds even for populations with infinite fourth moment.  Recently, \cite{li2024necessary} proposed a necessary and sufficient condition 
\begin{align} 
	x^3P(|X_1|>x)\to 0,\quad x\to\infty,\label{form:NSC_CLT}
\end{align}
under which the asymptotic distribution of linear spectral statistics is still normal. By the main result of \cite{li2024necessary}, the CLT \eqref{clt0} for $\tr(\R^2)$ holds under the condition \eqref{form:NSC_CLT} and p and n are of the same order.  But it is an open question whether there is a CLT and how it is dependent on $\alpha$ when $\alpha$ is less than three.

Note that nonparametric rank statistics such as
Kendall's tau and Spearman's rho perform well for continuous heavy-tailed distributions. They become much more involved when handling ties which are inherent in discrete data.
To the best of our knowledge, no theoretical results exist for high-dimensional
nonparametric statistics under discrete heavy-tailed distributions.

In view of the above, the aim of this work is to establish the CLT for Schott's statistic for general $\alpha$-regularly varying populations.  
It turns out that CLT of $\tr(\R^2)$ holds for general $\alpha>0$.
Roughly speaking, when $\alpha>3$, the asymptotic distribution  is $\alpha$-free and only depends on $n$ and $p$. When $\alpha<3$,  a phase transition happens where the leading term of the variance depends on $\alpha$. Therefore, we find a new renormalization constant dependent on $\alpha$ to ensure the validity of central limit theorem when $\alpha<3$. Following the seminal heavy-tailed random matrix works of \cite{belinschi2009free} and \cite{benaych2014spiked}, our work is the first to establish full-range CLT for spectral statistics of sample correlation matrices under heavy tails. Our proof does not rely on the framework of random matrix theory such that 
$p$ and $n$ are not required to be of comparable magnitude. In other words, we allow $p$ to be much larger than $n$ once $\alpha$ is large enough.

The overall proof strategy is as follows. Rewrite the sample correlation matrix as 
\begin{align*}
	\R= \begin{pmatrix}
		\|\X_1\|^{-1}&&\\
		&\ddots&\\
		& &\|\X_p\|^{-1}
	\end{pmatrix} \left(\Z_1,\cdots,\Z_n\right)   \begin{pmatrix}
		\Z_1\trans\\
		\vdots\\
		\Z_n\trans
	\end{pmatrix}\begin{pmatrix}
		\|\X_1\|^{-1}&&\\
		&\ddots&\\
		& &\|\X_p\|^{-1}
	\end{pmatrix}.
\end{align*}
Consider the Gram matrix of $\R$. By (\ref{1-April}) we have
\begin{align*}
	\B=\left(\X_1,\cdots,\X_p\right)  \begin{pmatrix}
		\|\X_1\|^2&&\\
		&\ddots&\\
		& &\|\X_p\|^2
	\end{pmatrix}^{-1} \begin{pmatrix}
		\X_1\trans\\
		\vdots\\
		\X_p\trans
	\end{pmatrix}=\sum_{i=1}^p \frac{\X_i\X_i\trans}{\|\X_i\|^2}.
\end{align*}
This is the so-called spatial-sign covariance matrix proposed by \cite{locantore1999robust}. A key observation is that $\tr(\R^2)=\tr(\B^2)$ so that the CLT for $\tr(\R^2)$ in \eqref{clt0} can be established by studying the linear spectral statistics of $\B$. 
When $\alpha<1$, the mean of the underlying random variables does not exist so that it is challenging to investigate the Schott statistic for non-symmetric heavy tailed distributions. Hence we turn to centralized Schott's statistic for general heavy tailed distributions. We use an important fact that the magnitudes of the maximum of i.i.d heavy tailed random variables and the sum of i.i.d heavy tailed random variables share the same order. We also make use of the exchangeability of self-normalized random variables and a key fact that the sum of the centralized self-normalized random variables is equal to zero. 


To the best of our knowledge, our work is the first one which systematically derives the asymptotic normality of $\tr(\R^2)$ for arbitrary $\alpha>0$. We also relax the constraint on the ratio of the data dimension $p$ and sample size $n$. In particular, the test statistics are applicable to high dimensional data with $p \to \infty$ 
if the underlying population has enough moments (e.g., $\alpha>5$) and the sample size $n$ tends to infinity. An important implication of our main results is that Schott's statistic after standardization still works well for heavy-tailed distributions, unlike nonparametric rank statistics such as Kendall's tau and Spearman's rho, which are designed for continuous distributions.
We summarize our contributions below:
\begin{enumerate}
	\item 
	We obtain the asymptotic normality for symmetric distributions and all $\alpha>0$,     \begin{align*}
		\frac{\tr(\R^2)-p-p(p-1)/n}{\sqrt{V_n(\alpha)}}\stackrel{d}{\to}N(0,1),
	\end{align*}
	where the explicit definition of $V_n(\alpha)$ is given in \eqref{def_variance} below.
	There exists a phase transition for the leading term of $V_n(\alpha)$ as $\alpha$ changes. When condition \eqref{form:NSC_CLT} holds, we recover the result of \cite{li2024necessary}, e.g.,  
	\begin{align*}
		\frac{n}{2p}\left(\tr(\R^2)-p-\frac{p(p-1)}{n}\right)\stackrel{d}{\to}N(0,1).
	\end{align*}   
	and when $\alpha<3$, $V_n(\alpha)$ depends on $\alpha$. Particularly, for $\alpha<2$, we have a neat new CLT as follows,
	\begin{align*}
		\frac{\sqrt{2n}}{(2-\alpha)p}\left(\tr(\R^2)-p-\frac{p(p-1)}{n}\right)\stackrel{d}{\to}N(0,1),~0<\alpha<2.
	\end{align*} 
	
	\item To ensure the asymptotic normality 
	the conditions about $n,p,\alpha$ are imposed as follows
	\begin{align*}
		\begin{cases}
			\frac{n}{p^2}\to 0,&\alpha<2,\\
			\frac{n^{\alpha-1+\epsilon}}{p^2}\to 0, &2\leq\alpha \leq 3,\\
			\frac{n^{5-\alpha+\epsilon}}{p^2}\to 0,&3<\alpha\leq 5,\\
			n, p \to \infty,&\alpha>5,
		\end{cases}
	\end{align*}
	for some $\epsilon>0$. Generally speaking,  the greater the value of $\alpha$, the less restriction on $p$. In particular, when $\alpha>5$, we only need $n, p \to\infty$ which means Schott's test is valid for high-dimensional data.
\item To accommodate unknown location parameters, we consider the centralized Pearson's sample correlation matrix $ \widetilde{\R}$ and obtain the following CLT for general distributions and all $\alpha>0$: 
\begin{align*}
\frac{\tr(\widetilde{\R}^2)-p-p(p-1)/(n-1)}{\sqrt{V_n(\alpha)}}\stackrel{d}{\to}N(0,1).
\end{align*}
This extends the existing CLTs to allow for unknown means 
and to the best of our knowledge, this is the first such result valid for all regularly varying indices $\alpha>0$.
\end{enumerate}

The following notations are used 
in the subsequent exposition. For a vector $\x$, $\|\x\|$ is the Frobenius norm. For two sequence of real positive integers $\{a_n\}$ and $\{b_n\}$, write $a_n \sim b_n$ if $a_n/b_n \to 1$.  Let $C, C_1,\ldots$ be a sequence of generic constants which may take different values at various places.

\section{Main results}\label{sec2}

\begin{assumption} Let $\{X_{ij},i=1,\ldots,p,j=1,\ldots,n\}$ be an array of i.i.d. random variables from the population $X$.  We assume that $X$ is regularly varying with index $\alpha>0$, i.e.,
\begin{align*}
P(|X|>x)\sim x^{-\alpha}l(x),\quad x\to\infty,
\end{align*}
where $l$ is a slowly varying function, i.e.
\begin{align*}
\frac{l(\lambda x)}{l(x)}\to 1,\quad x\to\infty,\quad\forall\lambda>0.
\end{align*}
\end{assumption}
We further define another slowly varying function
\begin{align} \label{def:tl}
\widetilde{l}(x)=\int_{0}^x l(t)/t dt,
\end{align}
where $l(x)/\widetilde{l}(x) \to 0$ and $a_n$ is a positive sequence satisfying $2na_n^{-1}\widetilde{l}(\sqrt{a_n})\to 1$. More details can be found in \cite{bingham1989regular}.  Throughout the paper, unless otherwise specified, we consider the joint asymptotic regime
\[
n\to\infty
\qquad\text{and}\qquad
p=p_n\to\infty.
\]
No specific limiting ratio between \(p\) and \(n\) is imposed; their relative growth rates are specified separately in the assumptions of each theorem.

Write
\begin{align*}
\X_i=\begin{pmatrix}
X_{i1}\\
\vdots\\
X_{in}
\end{pmatrix}, \quad i=1,\ldots,p.
\end{align*}
The Schott statistic can be rewritten as
\begin{align*}
T=\sum_{i,j=1}^p\frac{\left(\X_i\trans\X_j\right)^2}{\|\X_i\|^2\|\X_j\|^2}=p+\sum_{i\not=j}\frac{\left(\X_i\trans\X_j\right)^2}{\|\X_i\|^2\|\X_j\|^2}.
\end{align*}
We centralize the sample vectors for general distributions with unknown means, yielding the form of Schott's statistic
\begin{align*}
\widetilde{T}=p+\sum_{i\not=j}\frac{\left((\X_i-\bar{X}_i\one_n)\trans(\X_j-\bar{X}_j\one_n)\right)^2}{\|\X_i-\bar{X}_i\one_n\|^2 \|\X_j-\bar{X}_j\one_n\|^2},
\end{align*}
where 
\begin{align*}
\bar{X}_i=\frac{1}{n}\sum_{j=1}^n X_{ij},\quad \one_n=\begin{pmatrix}
1\\
\vdots\\
1
\end{pmatrix}\in\mR^n.
\end{align*}

\subsection{Symmetric distribution case}
This part is to consider the case when $X$ is symmetrically distributed, i.e.,
\begin{align*}
P(X>x)=P(X<-x),\quad \forall x\in\mR.
\end{align*}
Define the self-normalized variables
\begin{align*}
\Y_i=\begin{pmatrix}
Y_{i1}\\
\vdots\\
Y_{in}
\end{pmatrix}=\frac{\X_i}{\|\X_i\|}, i=1,\ldots,p.
\end{align*}
Schott's statistic is
\begin{align*}
T=\sum_{i,j=1}^p\left(\Y_i\trans\Y_j\right)^2=p+\sum_{i\not=j}\left(\Y_i\trans\Y_j\right)^2.
\end{align*}

Note that the expectation of odd powers of the self-normalized variables is zero in this case. 
As a consequence, the expectation and variance of Schott's statistics can be derived explicitly.
\begin{proposition}\label{thm:expectation_and_variance}
Assume that $X$ is a regularly varying variable with index $\alpha>0$ and $X \ind -X$.  We have
\begin{align}
\E(T)=&p+\frac{p(p-1)}{n},\label{form:expectation}\\
\var(T)=&2p(p-1) \left[\frac{n(n+2)}{n-1} \left(\E Y_{11}^4\right)^2-\frac{6}{n-1}\E Y_{11}^4+\frac{2n+1}{n^2(n-1)} \right].\label{form:variance}
\end{align}
\end{proposition}

As can be seen, the expectation remains independent of $\alpha$, while the variance is determined by $\E Y_{11}^4$, which is influenced by $\alpha$. To be more precise, by Proposition \ref{prop:even_power} below, we can obtain
\begin{align*}
\E Y_{11}^4\sim\begin{cases}
(1-\frac{\alpha}{2})n^{-1},&\alpha<2,\\
l(\sqrt{a_n})a_n^{-1},&\alpha=2,~\E X^2=\infty,\\
\frac{\alpha\Gamma(2-\frac{\alpha}{2})\Gamma(\frac{\alpha}{2})}{2}l(\sqrt{n})n^{-\frac{\alpha}{2}},&2 \leq \alpha<4, ~\E X^2<\infty,\\
4\widetilde{l}(\sqrt{n})n^{-2},&\alpha=4,~\E X^{4}=\infty,\\
\E X_1^{4} n^{-2},&\alpha\geq 4,~\E X^{4}<\infty,
\end{cases}
\end{align*}
where $\Gamma()$ is the Gamma function and $\tilde{l}()$ is the slowly varying function defined  in \eqref{def:tl}. As shown in \eqref{form:variance}, the limiting behavior of variance $\var(T)$ depends on the scale of $n\left(\E Y_{11}^4\right)^2$ and a phase transition occurs when $\alpha=3$. When $\alpha>3$, only the $2/n^2$ term matters, whereas for $\alpha\leq 3$, the $n\left(\E Y_{11}^4\right)^2$ term contributes, and its effect varies depending on the value of $\alpha$. 
The leading terms of $\var(T)$ are summarized as follows
\begin{align} \label{def_variance}
V_n(\alpha)=&\begin{cases}
2(1-\alpha/2)^2p^2/n,&\alpha<2,\\
l^2(\sqrt{a_n})p^2/(2n\widetilde{l}^2(\sqrt{a_n})),&\alpha=2,\\
\alpha^2\Gamma^2(2-\alpha/2)\Gamma^2(\alpha/2)l^2(\sqrt{n})p^2/\left(2 n^{\alpha-1}\right),&2<\alpha<3,\\
4p^2/n^2+9\pi^2l^2(\sqrt{n})p^2/(8 n^2),&\alpha=3,\\
4p^2/n^2,&\alpha>3.
\end{cases}
\end{align}

We next present 
the CLT of Schott's statistic when $X$ is symmetrically distributed.
\begin{theorem}\label{thm:CLT}
Assume that $X$ is a regularly varying variable with index $\alpha>0$, $X \ind -X$ and
\begin{align} \label{maincondition}
\begin{cases}
	\frac{n}{p^2}\to 0,&0<\alpha<2,\\
	\frac{n^{\alpha-1}}{p^2} \frac{\tilde{l}^2(\sqrt{a_n})}{l^2(\sqrt{a_n})}  \to 0, &\alpha=2,~\E X^{2}=\infty,\\
	\frac{n^{\alpha-1}}{p^2} \frac{1}{l^2(\sqrt{n})}  \to 0,&2\leq\alpha<3,~\E X^{2}<\infty,\\
	\frac{n}{p}\frac{l(\sqrt{n})}{l^2(\sqrt{n})+2}\to 0,&\alpha=3,\\
	\frac{n^{5-\alpha}l^2(\sqrt{n})}{p^2}\to 0,&3<\alpha\leq 5,\\
	n, p \to \infty,&\alpha>5.
\end{cases}
\end{align}
We have  
\begin{align*}
\frac{\tr(\R^2)-p-p(p-1)/n}{\sqrt{V_n(\alpha)}}\stackrel{d}{\to}N(0,1).
\end{align*}
\end{theorem}

A phase transition occurs at $\alpha=3$. When $\alpha>3$ or $\alpha=3$ with $l(x)\to 0$, we have 
\begin{align*}
V_n(\alpha)\sim 4p^2/n^2,
\end{align*} 
such that
\begin{align*}
\frac{n}{2p}\left(\tr(\R^2)-p-\frac{p(p-1)}{n}\right)\stackrel{d}{\to}N(0,1).
\end{align*}  
Observe that $\alpha>3$ and $\alpha=3$ with $l(x)\to 0$ imply that 
\begin{align*} 
x^3P(|X_1|>x)\to 0,\quad x\to\infty.
\end{align*}
Our result is consistent with that in \cite{li2024necessary} in this case. In particular, \cite{li2024necessary} proved that the above condition is necessary and sufficient for general linear spectral statistics and the asymptotic CLT depends on the positive ratio $p/n$ only. Here we relax the constraint on the ratio of $p/n$ such that the ratio could either tend to zero or infinity. 

When $\alpha<3$, the variance $V_n(\alpha)$ is dominated by $2p^2 n\left(\E Y_{11}^4\right)^2$ and we need a new renormalization constant to ensure the validity of CLT. For $\alpha<3$, the variance $V_n(\alpha)$ increases as $\alpha$ decreases, indicating that heavier tails lead to larger fluctuations in Schott's statistic. For $2<\alpha<3$, the CLT is complex and for $\alpha<2$, we have a neat new CLT as follows,
\begin{align*}
\frac{\sqrt{2n}}{(2-\alpha)p}\left(\tr(\R^2)-p-\frac{p(p-1)}{n}\right)\stackrel{d}{\to}N(0,1),~0<\alpha<2.
\end{align*} 

\begin{remark}[Ratio condition]
To remove the slowly varying functions involved in (\ref{maincondition}), we can roughly impose the following ratio conditions between $p$ and $n$,   
\begin{align*}
\begin{cases}
	\frac{n}{p^2}\to 0,&\alpha<2,\\
	\frac{n^{\alpha-1+\epsilon}}{p^2}\to 0, &2\leq\alpha \leq 3,\\
	\frac{n^{5-\alpha+\epsilon}}{p^2}\to 0,&3<\alpha\leq 5,\\
	n, p \to \infty, &\alpha>5,
\end{cases}
\end{align*}
for some $\epsilon>0$. Generally speaking, the greater the value of $\alpha$, the weaker restriction on $p$. For example, when $\alpha>5$, we only need $n, p \to\infty$ which means Schott's test is valid for high dimensional data with $p \to \infty$. If we view Schott's statistics as a U-statistic of order two based on the sample correlation matrix, our results can be viewed as a kind of an extension of \cite{chen2010tests} which studied the U-statistics based on sample covariance matrices. 
\end{remark}

\subsection{General distribution cases for centralized data} 
This section is devoted to the CLT for Schott's statistic for centralized data. 
We first define the centralized self-normalized variables 
as follows
\begin{align}
\tbY_i=\frac{\X_i-\bar{X}_i\one_n}{\|\X_i-\bar{X}_i\one_n\|}=\frac{\Y_i-\bar{Y}_i\one_n}{\sqrt{1-n\bar{Y}_i^2}}, \quad i=1,\ldots,p,\label{form:def_Y_tilde}
\end{align}
where 
\begin{align}
\bar{X}_i=\frac{1}{n}\one_n^\top \X_i=\frac{1}{n}\sum_{k=1}^n X_{ik}, \quad \bar{Y}_i=\frac{1}{n}\one_n^\top \Y_i=\frac{1}{n}\sum_{k=1}^n Y_{ik},\label{form:def_Y_bar}
\end{align}
and 
\begin{align*}
\tY_{ik}=\frac{X_{ki}-\bar{X}_i}{\sqrt{\sum_{k=1}^nX_{ki}^2-n\bar{X}_i^2}}=\frac{Y_{ik}-\bar{Y}_i}{\sqrt{1-n\bar{Y}_i^2}}.
\end{align*}
The centralized Pearson's sample correlation matrix is defined as 
\begin{align*}
\widetilde{\R}=\left(\widetilde{r}_{ij}\right)_{p \times p},\quad \widetilde{r}_{ij}=\frac{\sum_{k=1}^nX_{ki}X_{kj}-n\bar{X}_i\bar{X}_j}{\sqrt{\left(\sum_{k=1}^nX_{ki}^2-n\bar{X}_i^2\right)\left(\sum_{k=1}^nX_{kj}^2-n\bar{X}_j^2\right)}}.
\end{align*}
In the form of self-normalized variables $\widetilde{r}_{ij}$ can be also written as
\begin{align*}
\widetilde{r}_{ij}=\frac{\sum_{k=1}^nY_{ik}Y_{jk}-n\bar{Y}_i\bar{Y}_j}{\sqrt{\left(1-n\bar{Y}_i^2\right)\left(1-n\bar{Y}_j^2\right)}}=\sum_{k=1}^n\tY_{ik}\tY_{jk}.
\end{align*}
Thus the centralized Schott statistic is 
\begin{align*}  
\widetilde{T}=\tr(\widetilde{\R}^2)=\sum_{i,j=1}^p\left(\widetilde{r}_{ij}\right)^2=p+\sum_{i\not=j}\left(\sum_{k=1}^n\tY_{ik}\tY_{jk}\right)^2=p+\sum_{i\not=j}\left(\tbY_i^\top \tbY_j\right)^2.
\end{align*}

Note that $\tY_{i1},\cdots,\tY_{in}$ are identically distributed random variables and exchangeable. 
A key observation is that
\begin{align*}
\sum_{k=1}^n \tY_{ik}=0,\quad \sum_{k=1}^n \tY_{ik}^2=1.
\end{align*}
Hence we have the trivial moment results
\begin{gather*}
\E \tY_1=0,\quad \E \tY_1^2=\frac{1}{n},\quad \E \left( \tY_1\tY_2\right)=-\frac{1}{n-1}\E \tY_1^2=-\frac{1}{n(n-1)},\\
\E \left(\tY_1^2\tY_2^2\right)=\frac{1}{n(n-1)}-\frac{1}{n-1}\E \tY_1^4, \quad \E \left(\tY_1^3\tY_2\right)=-\frac{1}{n-1}\E \tY_1^4.
\end{gather*}
Due to this elegant property we have for any $\alpha>0$,
\begin{align*}
\E\tr(\widetilde{\R}^2)=p+\frac{p(p-1)}{n-1}.
\end{align*}
We are now in a position to state CLT for the centralized data.
\begin{theorem}\label{thm:CLT3}
Suppose that $X$ is a regularly varying variable with index $\alpha>0$ and 
\begin{align} \label{maincondition-general}
\begin{cases}
\frac{n}{p^2}\to 0,&0<\alpha<2,\\
\frac{n^{\alpha-1}}{p^2} \frac{\tilde{l}^2(\sqrt{a_n})}{l^2(\sqrt{a_n})}  \to 0, &\alpha=2,~\E X^{2}=\infty,\\
\frac{n^{\alpha-1}}{p^2} \frac{1}{l^2(\sqrt{n})}  \to 0,&2\leq\alpha<3,~\E X^{2}<\infty,\\
\frac{n}{p}\frac{l(\sqrt{n})}{l^2(\sqrt{n})+2}\to 0,&\alpha=3,\\
\frac{n^{5-\alpha}l^2(\sqrt{n})}{p^2}\to 0,&3<\alpha\leq 5,\\
n, p \to \infty,&\alpha>5.
\end{cases}
\end{align} 
We have  
\begin{align*}
\frac{\tr(\widetilde{\R}^2)-p-p(p-1)/(n-1)}{\sqrt{V_n(\alpha)}}\stackrel{d}{\to}N(0,1).
\end{align*}
\end{theorem}

Compared with Theorem \ref{thm:CLT}, the only difference is the asymptotic mean 
and there we use $n-1$ to replace $n$. This is another instance of the substitution principle \citep{zheng2015substitution}. 

\subsection{Consistent estimation for the asymptotic variance}
We need to estimate the asymptotic variance $V_n(\alpha)$ to implement the above CLT results in practice. However, it is challenging to estimate $V_n(\alpha)$ directly since it depends on the regularly varying index $\alpha$ and the slowly varying function $l(x)$, both of which are unknown in advance. In extreme value theory, estimating the regularly varying index $\alpha$ has been extensively studied.  For instance, two commonly used estimators are the Hill estimator \cite{hill1975simple} and the Pickands estimator \citep{pickands1975statistical}. 

Instead of estimating the variance through estimating $\alpha$ and $l(x)$ we turn to directly construct an estimator for $\E Y_{11}^4$ as
\begin{align*}
\frac{1}{np}\sum_{i=1}^p\sum_{j=1}^n \tY_{ij}^4.
\end{align*}
Plugging it into \eqref{form:variance} yields 
\begin{align*}
\widehat{V}=\frac{2}{n}\left(\sum_{i=1}^p\sum_{j=1}^n \tY_{ij}^4\right)^2+\frac{4p^2}{n^2}.
\end{align*}

\begin{theorem}\label{thm:CLT_distribution_free}
Under the assumptions of Theorem \ref{thm:CLT3}, we have
\begin{align*}
\frac{\tr(\widetilde{\R}^2)-p-p(p-1)/(n-1)}{\sqrt{ \widehat{V}}} \stackrel{d}{\to}N(0,1).
\end{align*}
\end{theorem}

\section{Application}\label{sec5}
An important application of correlation matrices is to test mutual independence of a p-dimensional random vector $\X=\left(X_1,\cdots,X_p\right)\trans$,
\begin{align*}
H_0:X_1,\cdots,X_p\text{ are independent}.
\end{align*}
Based on the empirical correlation matrices such as Pearson, Spearman and Kendall's correlation matrices, we can consider the test statistics based on the maximum norm or the Frobenius norm of the correlation matrices. In particular, Schott's statistic is based on the Frobenius norm of  Pearson's correlation matrix. Existing studies on testing correlation structures can be categorized according to the choice of the norm and the type of the correlation measure. 
For the Frobenius norm, results for the sample correlation matrix are given by \cite{gao2017high,zheng2019test}; those for Kendall’s tau are established in \cite{leung2018testing, li2021central}; and the results for Spearman’s rho are derived in \cite{bao2015spectral, leung2018testing, chen2024large}. For the maximum norm, the sample correlation matrix is studied in \cite{zhou2007asymptotic}, while both Kendall’s tau and Spearman’s rho are investigated in \cite{han2017distribution}.

By Theorem \ref{thm:CLT_distribution_free} we can construct a distribution-free test statistic based on Pearson's correlation matrix as follows,
\begin{align*}
\frac{\tr(\widetilde{\R}^2)-p-p(p-1)/(n-1)}{\sqrt{ \widehat{V}}} \stackrel{d}{\to}N(0,1),
\end{align*}
which is valid for regularly varying index $\alpha>0$.

\subsection{Continuous heavy-tailed distribution}
To examine the finite sample performance of our proposed statistics, we conduct null hypotheses from the following types of populations:
\begin{itemize}
\item[(1).] Student's t-distribution: $X_{ij}$ are i.i.d. $t(\alpha)$ for $1\leq i\leq n$ and $1\leq j\leq p$.
\item[(2).] Pareto distribution: $X_{ij}$ are i.i.d. $\mathrm{Pareto}(\alpha)$ for $1\leq i\leq n$ and $1\leq j\leq p$.
\end{itemize}
In details, Student's t-distribution is a symmetric population and Pareto distribution is asymmetric. As for numerical experiments, we take different combinations of sample size $n$, data dimension $p$ and regularly varying index $\alpha$. We adopt the data dimensionality settings from \cite{chen2010tests}, where the feature dimension $p$ satisfies $p=c\exp(n^\eta)$. For sample sizes $n=20,40,60,80$, the corresponding feature dimensions are $p=35,62,91,124$ for $(c,\eta)=(3,0.3)$ and $p=55,159,343,642$ for $(c,\eta)=(2,0.4)$.

Table \ref{table:emp_size} shows the empirical sizes of test statistics at a nominal level of $5\%$ based on 1000 replications. For comparison, we denote the maximum norm type statistics based on Pearson's correlation matrix, Kendall's tau and Spearman's rho, denoted as $L_{\R,\max},L_{\K,\max},L_{\brho,\max}$, respectively. The Frobenius norm type statistics are denoted as $L_{\R,2},L_{\K,2},L_{\brho,2}$, respectively. While the existing CLT of Schott's statistic is valid for light-tailed populations, our new CLT can obtain satisfactory sizes for all cases. Our results do not rely on the common framework of random matrix theory where $p/n \to c \in(0,+\infty)$, and the asymptotic distributions remain valid for high-dimensional data ($p/n \to +\infty$). 
\begin{table}[ht!]
\centering  
\caption{Empirical sizes of independence test statistics based on Pearson, Spearman and Kendall's correlations.}
\label{table:emp_size}
\begin{tabular}{@{\extracolsep{5pt}} ccccccccc} 
\\[-1.8ex]\hline 
\hline
$n$ & 20 & 40 & 60 & 80 & 20 & 40 & 60 & 80\\
\hline
$p$ & 35 & 62 & 91 & 124 & 55 & 159 & 343 & 642\\
\hline\\
\multicolumn{9}{c}{Normal distribution (t-distribution with $\alpha=+\infty$)}\\
\mbox{New} & 0.046 & 0.044 &0.044 & 0.047 & 0.047 & 0.031 & 0.036 & 0.049 \\ 
$L_{\R,2}$ & 0.070 & 0.050 & 0.048 & 0.054 & 0.053 & 0.046 & 0.044 & 0.057 \\ 
$L_{\R,max}$ & 0.000 & 0.004 & 0.013 & 0.013 & 0.000 & 0.000 & 0.002 & 0.003 \\ 
$L_{\brho,2}$ & 0.059 & 0.050 & 0.052 & 0.051 & 0.046 & 0.046 & 0.054 & 0.053 \\ 
$L_{\brho,max}$ & 0.000 & 0.003 & 0.016 & 0.021 & 0.000 & 0.001 & 0.003 & 0.004 \\ 
$L_{\K,2}$ & 0.086 & 0.065 & 0.061 & 0.063 & 0.082 & 0.061 & 0.057 & 0.061 \\ 
$L_{\K,max}$ & 0.013 & 0.015 & 0.035 & 0.034 & 0.005 & 0.014 & 0.017 & 0.010 \\ 
\hline \\
\multicolumn{9}{c}{t-distribution with $\alpha=1$}\\
\mbox{New} & 0.040 & 0.052 & 0.056 & 0.067 & 0.047 & 0.046 & 0.029 & 0.059 \\ 
$L_{\R,2}$ & 0.142 & 0.215 & 0.255 & 0.297 & 0.158 & 0.229 & 0.256 & 0.290 \\ 
$L_{\R,max}$ & 0.947 & 1.000 & 1.000 & 1.000 & 0.967 & 1.000 & 1.000 & 1.000 \\ 
$L_{\brho,2}$ & 0.054 & 0.056 & 0.061 & 0.036 & 0.056 & 0.049 & 0.055 & 0.040 \\ 
$L_{\brho,max}$ & 0.000 & 0.010 & 0.012 & 0.014 & 0.000 & 0.001 & 0.004 & 0.006 \\ 
$L_{\K,2}$ & 0.066 & 0.052 & 0.059 & 0.077 & 0.069 & 0.080 & 0.061 & 0.054 \\ 
$L_{\K,max}$ & 0.004 & 0.025 & 0.027 & 0.026 & 0.008 & 0.015 & 0.018 & 0.021 \\ 
\hline \\
\multicolumn{9}{c}{Pareto distribution with $\alpha=1$}\\
\mbox{New} & 0.053 & 0.046 & 0.044 & 0.066 & 0.053 & 0.053 & 0.037 & 0.055 \\ 
$L_{\R,2}$ & 0.174 & 0.230 & 0.248 & 0.300 & 0.189 & 0.248 & 0.263 & 0.296 \\ 
$L_{\R,max}$ & 0.983 & 1.000 & 1.000 & 1.000 & 0.994 & 1.000 & 1.000 & 1.000 \\ 
$L_{\brho,2}$ & 0.047 & 0.049 & 0.064 & 0.059 & 0.042 & 0.050 & 0.044 & 0.063 \\ 
$L_{\brho,max}$ & 0.000 & 0.004 & 0.013 & 0.013 & 0.000 & 0.004 & 0.006 & 0.001 \\ 
$L_{\K,2}$ & 0.086 & 0.066 & 0.058 & 0.056 & 0.074 & 0.069 & 0.067 & 0.052 \\ 
$L_{\K,max}$ & 0.012 & 0.019 & 0.026 & 0.024 & 0.008 & 0.015 & 0.017 & 0.017 \\ 
\hline \\
[-1.8ex]
\end{tabular} 
\end{table} 

\subsection{Discrete heavy-tailed distribution}
While nonparametric rank statistics such as Spearman’s $\brho$ and Kendall's $\K$ perform well for continuous heavy-tailed distributions, they rely on the continuity assumption and thus fail to account for ties, which are inherent in discrete data. To the best of our knowledge, no theoretical results exist for high-dimensional nonparametric statistics under discrete heavy-tailed distributions. In contrast, our analysis of Schott’s statistic for the Pearson's correlation matrix imposes no continuity assumptions and is therefore directly applicable to both continuous and discrete heavy-tailed distributions.

In this part, we consider Zipf distribution which is a discrete heavy-tailed distribution supported on positive integers.
For a tail index $\alpha>0$ and truncation $N$, its probability mass function is
\begin{align*}
\mathbb{P}(X=k) = \frac{1}{k^{\alpha+1} H_{N}(\alpha)}, \qquad k=1,2,\dots,N,
\end{align*}
where $H_{N}(\alpha) = \sum_{k=1}^N k^{-\alpha-1}$ denotes the generalized harmonic number.
As $N\to\infty$ and $\alpha>0$, its tail follows a power law
\begin{align*}
\mathbb{P}(X>x) \sim C x^{-\alpha}, \quad x\to\infty,
\end{align*}
indicating regular variation with index $\alpha$.

Table \ref{table:discrete_emp_size} presents
the empirical sizes of the tests at a nominal significance level of $5\%$ based on 1000 replications. For different values of the heavy-tailed index $\alpha$, we find that the empirical sizes are consistently close to $5\%$. This fully demonstrates the validity of our proposed statistic for discrete heavy-tailed variables.

\begin{table}[ht!]
\centering  
\caption{Empirical sizes of the proposed statistic under discrete distributions.}
\label{table:discrete_emp_size}
\begin{tabular}{@{\extracolsep{5pt}} ccccccccc} 
\\[-1.8ex]\hline 
\hline
$n$ & 20 & 40 & 60 & 80 & 20 & 40 & 60 & 80\\
\hline
$p$ & 35 & 62 & 91 & 124 & 55 & 159 & 343 & 642\\
\hline\\
$\alpha=0.5$ & 0.048 & 0.059 & 0.064 & 0.055 & 0.050 & 0.048 & 0.049 & 0.050 \\ 
$\alpha=1$ & 0.055 & 0.043 & 0.051 & 0.053 & 0.048 & 0.071 & 0.044 & 0.060 \\ 
$\alpha=1.5$ & 0.056 & 0.054 & 0.053 & 0.042 & 0.053 & 0.062 & 0.042 & 0.056 \\ 
$\alpha=2$ & 0.053 & 0.037 & 0.043 & 0.053 & 0.044 & 0.032 & 0.048 & 0.057 \\ 
\hline \\
[-1.8ex]
\end{tabular} 
\end{table} 

\newpage
	\section*{Appendix}
	
	\setcounter{section}{0}
	\setcounter{equation}{0}
	\def\theequation{A.\arabic{equation}}
	\def\thesection{A\arabic{section}}
	This Appendix contains all supporting lemmas and proofs. 

\section{Asymptotic results for self-normalized variables}
For i.i.d. random variables $X_1,\cdots,X_n$, we define the self-normalized variables $Y_1,\cdots,Y_n$
\begin{align*}
Y_i=\frac{X_i}{\sqrt{X_1^2+\cdots+X_n^2}},\quad 1\leq i\leq n.
\end{align*}
By the identity of Gamma function
\begin{align*}
\frac{1}{x^\beta}=\frac{1}{\Gamma(\beta)}\int_0^\infty s^{\beta-1}e^{-sx}ds,
\end{align*}
and Fubini's theorem, we obtain 
\begin{align*}
\E \left( Y_1^{k_1}\cdots Y_r^{k_r}\right)=\frac{1}{\Gamma(\frac{k}{2})}\int_0^\infty s^{\frac{k}{2}-1}\prod_{i=1}^r\E\left( X_i^{k_i}e^{-sX_i^2}\right)\left(\E e^{-sX_1^2}\right)^{n-r}ds,
\end{align*}
where $k=k_1+\cdots+k_r$. Asymptotic behavior concerning the moments of the self-normalized variables depends on the Laplace transformation $\varphi(s)=\E e^{-sX^2}$. See \cite{albrecher2007asymptotic} for more details.

We first collect two important lemmas about the Laplace transformation of heavy-tailed distributions and then present the main results of self-normalized variables.
\begin{lemma}[Corollary 8.1.7, Theorem 8.8.1 of \citealt{bingham1989regular} and Page 7 of \citealt{albrecher2007asymptotic}]\label{lem:exact_limiting_phi}
Assume that $P(|X|>x)\sim x^{-\alpha}l(x)$ where $l$ is a slowly varying function. Let $\varphi(s)=\E e^{-sX^2}$, then
\begin{align*}
	1-\varphi(s)\sim\begin{cases}
		\Gamma(1-\frac{\alpha}{2})s^{\frac{\alpha}{2}}l(\frac{1}{\sqrt{s}}),&\alpha<2,\\
		2s\widetilde{l}(\frac{1}{\sqrt{s}}),&\alpha=2,\E X^{2}=\infty,\\
		\E X^{2}s,&\alpha\geq2,\E X^{2}<\infty,
	\end{cases}\quad(s\to 0^+)
\end{align*}
and
\begin{align*}
	(-1)^k\varphi^{(k)}(s) \sim\begin{cases}
		\frac{\alpha}{2}\Gamma(k-\frac{\alpha}{2})s^{\frac{\alpha}{2}-k}l(\frac{1}{\sqrt{s}}),&2k>\alpha,\\
		2k\widetilde{l}(\frac{1}{\sqrt{s}}),&2k=\alpha,\E X^{2k}=\infty,\\
		\E X^{2k},&2k\leq\alpha,\E X^{2k}<\infty,
	\end{cases}\quad(s\to 0^+)
\end{align*}
where $\varphi^{(k)}(s)=\frac{d^k}{d s^k} \varphi(s)$. Here
$\widetilde{l}(x)=\int_{0}^x l(t)/t dt$ is itself a slowly varying function and $l(x)/\widetilde{l}(x) \to 0$.
\end{lemma}

\begin{lemma}[Pages 349 and 373 of \citealt{bingham1989regular}]\label{lem:norming_constants}
Let $F$ be a distribution with positive support, and $\varphi$ be its Laplace transformation. For any $\beta \in(0,1]$, if 
\begin{align*}
	1-\varphi(s)\sim s^{\beta}l(\frac{1}{s}),\quad s\to 0^+,
\end{align*}
then, there exists a positive sequence $(a_n)_{n=1}^\infty$, $a_n\to\infty$, such that
\begin{align*}
	na_n^{-\beta}l(a_n)\to 1.
\end{align*}
\end{lemma}

\begin{proposition}\label{prop:even_power}
Assume that $X_1,X_2,\cdots$ are i.i.d. with $P(|X_1|>x)\sim x^{-\alpha}l(x)$ where $l$ is a slowly varying function. For $k\geq 2$, we have 
\begin{align*}
	\E Y_1^{2k}\sim \begin{cases}
		\frac{\Gamma(k-\frac{\alpha}{2})}{\Gamma(k)\Gamma(1-\frac{\alpha}{2})}n^{-1},&\alpha<2,\\
		\frac{l(\sqrt{a_n})}{a_n(k-1)} ,&\alpha=2,~\E X_1^2=\infty,\\
		\frac{\alpha\Gamma(k-\frac{\alpha}{2})\Gamma(\frac{\alpha}{2})}{2\Gamma(k)}l(\sqrt{n})n^{-\frac{\alpha}{2}},&2 \leq \alpha<2k, ~\E X_1^2<\infty,\\
		2k \widetilde{l}(\sqrt{n})n^{-k},&\alpha=2k,~\E X_1^{2k}=\infty,\\
		\E X_1^{2k} n^{-k},&\alpha\geq 2k,~\E X_1^{2k}<\infty,
	\end{cases}
\end{align*}
where $a_n$ is a positive sequence satisfying $2na_n^{-1}\widetilde{l}(\sqrt{a_n})\to 1$. In all cases where \(\E X_1^2<\infty\), we normalize \(X_1\) such that \(\E X_1^2=1\).
\end{proposition}

\begin{proof} Following \cite{albrecher2007asymptotic}, the expectation can be represented as 
\begin{align*}
	\E Y_1^{2k}=\frac{(-1)^k}{\Gamma(k)}\int_0^\infty s^{k-1}\varphi^{(k)}(s)\varphi^{n-1}(s)ds,
\end{align*}
where $\varphi(s)=\E e^{-sX_1^2}$.

If $0<\alpha<2$, by Lemma \ref{lem:norming_constants}, we can choose $a_n \to \infty$ such that 
\begin{align*}
	na_n^{-\frac{\alpha}{2}}l(\sqrt{a_n})\Gamma(1-\frac{\alpha}{2})\to 1,
\end{align*}
and by Lemma \ref{lem:exact_limiting_phi},
\begin{align*}
	\varphi^n(t/a_n)=\exp\left\{n\log\varphi(t/a_n)\right\}\sim \exp\left\{-n\Gamma(1-\alpha/2)\left(\frac{t}{a_n}\right)^{\frac{\alpha}{2}}l(\sqrt{a_n/t})\right\}\sim\exp\left\{-t^{\frac{\alpha}{2}}\right\}.
\end{align*}
This, together with Lemma \ref{lem:exact_limiting_phi}, implies that  
\begin{align*}
	&\frac{(-1)^k}{\Gamma(k)}\int_0^\infty s^{k-1}\varphi^{(k)}(s)\varphi^{n-1}(s)ds\\
	=&\frac{(-1)^k}{\Gamma(k)}\int_0^\infty \left(\frac{t}{a_n}\right)^{k-1}\varphi^{\left(k\right)}\left(a^{-1}_n t\right)\varphi^{n-1}\left(a_n^{-1}t\right)\frac{1}{a_n}dt\\
	\sim& \frac{1}{\Gamma(k)}\int_0^\infty \frac{t^{k-1}}{a_n^k}\left[\frac{\alpha}{2}\Gamma(k-\frac{\alpha}{2})(\frac{t}{a_n})^{\frac{\alpha}{2}-k}l(\frac{\sqrt{a_n}}{\sqrt{t}})\right]\varphi^n(\frac{t}{a_n})dt\\
	=&\frac{\alpha\Gamma(k-\frac{\alpha}{2})}{2\Gamma(k)\Gamma(1-\frac{\alpha}{2})}\cdot\frac{l(\sqrt{a_n})\Gamma(1-\frac{\alpha}{2})}{a_n^{\frac{\alpha}{2}}}\int_0^\infty t^{\frac{\alpha}{2}-1}\frac{l(\sqrt{a_n}/\sqrt{t})}{l(\sqrt{a_n})}\varphi^n(\frac{t}{a_n})dt\\
	\sim& \frac{\alpha\Gamma(k-\frac{\alpha}{2})}{2n\Gamma(k)\Gamma(1-\frac{\alpha}{2})}\int_0^\infty t^{\frac{\alpha}{2}-1}\exp\left\{-t^{\frac{\alpha}{2}}\right\}dt\\
	=&\frac{\Gamma(k-\frac{\alpha}{2})}{n\Gamma(k)\Gamma(1-\frac{\alpha}{2})}.
\end{align*}
By Potter’s bounds, the integral representation can be dominated by an integrable function after splitting the integral into \(t\le1\) and \(t>1\); hence dominated convergence applies.

If $\alpha=2$ and $\E X_1^2=\infty$, by Lemma \ref{lem:exact_limiting_phi} and Lemma \ref{lem:norming_constants}, we can choose $a_n$ such that 
\begin{align*}
	2na_n^{-1}\widetilde{l}(\sqrt{a_n})\to 1,
\end{align*}
and 
\begin{align*}
	\varphi^n(t/a_n)=\exp\left\{n\log\varphi(t/a_n)\right\}\sim \exp\left\{-2n\frac{t}{a_n}\widetilde{l}(\sqrt{a_n/t})\right\}\sim e^{-t}.
\end{align*}
And by Lemma \ref{lem:exact_limiting_phi}, we have 
\begin{align*}
	&\frac{(-1)^k}{\Gamma(k)}\int_0^\infty s^{k-1}\varphi^{(k)}(s)\varphi^{n-1}(s)ds\\
	=&\frac{(-1)^k}{\Gamma(k)}\int_0^\infty (\frac{t}{a_n})^{k-1}\varphi^{(k)}(\frac{t}{a_n})\varphi^{n-1}(\frac{t}{a_n})\frac{1}{a_n}dt\\
	\sim& \frac{1}{\Gamma(k)}\int_0^\infty \frac{t^{k-1}}{a_n^k}\left[\frac{\Gamma(k-1)t^{1-k}l(\frac{\sqrt{a_n}}{\sqrt{t}})}{a_n^{1-k}}\right]e^{-t}dt\\
	\sim& \frac{\Gamma(k-1)}{\Gamma(k)}  \frac{l(\sqrt{a_n})}{a_n}\int_0^\infty e^{-t}dt=\frac{l(\sqrt{a_n})}{(k-1)a_n}.
\end{align*}

When $\E X_1^2<\infty$ (without loss of generality, we let $\E X_1^2=1$) and $\alpha<2k$, by Lemma \ref{lem:exact_limiting_phi},
\begin{align*}
	&\frac{(-1)^k}{\Gamma(k)}\int_0^\infty s^{k-1}\varphi^{(k)}(s)\varphi^{n-1}(s)ds\\
	=&\frac{(-1)^k}{\Gamma(k)}\int_0^\infty (\frac{t}{n})^{k-1}\varphi^{(k)}(\frac{t}{n})\varphi^{n-1}(\frac{t}{n})\frac{1}{n}dt\\
	\sim&\frac{1}{\Gamma(k)}\int_0^\infty\frac{t^{k-1}}{n^k}\left[\frac{\alpha}{2}\Gamma(k-\frac{\alpha}{2})(\frac{t}{n})^{\frac{\alpha}{2}-k}l(\frac{\sqrt{n}}{\sqrt{t}})\right]e^{-t}dt\\
	\sim&\frac{\alpha\Gamma(k-\frac{\alpha}{2})l(\sqrt{n})}{2\Gamma(k)n^{\frac{\alpha}{2}}}\int_0^\infty t^{\frac{\alpha}{2}-1}e^{-t}dt\\
	=&\frac{\alpha\Gamma(k-\frac{\alpha}{2})\Gamma(\frac{\alpha}{2})}{2\Gamma(k)}n^{-\frac{\alpha}{2}}l(\sqrt{n}).
\end{align*}

When $\alpha=2k$ and $\E X_1^{2k}=\infty$, by Lemma \ref{lem:exact_limiting_phi},
\begin{align*}
	&\frac{(-1)^k}{\Gamma(k)}\int_0^\infty s^{k-1}\varphi^{(k)}(s)\varphi^{n-1}(s)ds\\
	=&\frac{(-1)^k}{\Gamma(k)}\int_0^\infty (\frac{t}{n})^{k-1}\varphi^{(k)}(\frac{t}{n})\varphi^{n-1}(\frac{t}{n})\frac{1}{n}dt\\
	\sim&\frac{1}{\Gamma(k)}\int_0^\infty\frac{t^{k-1}}{n^k}\left[2k\widetilde{l}(\frac{\sqrt{n}}{\sqrt{t}})\right]e^{-t}dt\\
	\sim&\frac{2k\widetilde{l}(\sqrt{n})}{\Gamma(k)n^k}\int_0^\infty t^{k-1}e^{-t}dt\\
	=&2kn^{-k}\widetilde{l}(\sqrt{n}).  
\end{align*}

When $\alpha\geq 2k$ and $\E X_1^{2k}<\infty$, by Lemma \ref{lem:exact_limiting_phi},
\begin{align*}
	&\frac{(-1)^k}{\Gamma(k)}\int_0^\infty s^{k-1}\varphi^{(k)}(s)\varphi^{n-1}(s)ds\\
	=&\frac{(-1)^k}{\Gamma(k)}\int_0^\infty (\frac{t}{n})^{k-1}\varphi^{(k)}(\frac{t}{n})\varphi^{n-1}(\frac{t}{n})\frac{1}{n}dt\\
	\sim&\frac{1}{\Gamma(k)}\int_0^\infty\frac{t^{k-1}}{n^k}\E X_1^{2k}e^{-t}dt\\
	=&\E X_1^{2k}n^{-k}.
\end{align*}

Collecting all the terms, we can conclude that
\begin{align*}
	\E Y_1^{2k}\sim \begin{cases}
		\frac{\Gamma(k-\frac{\alpha}{2})}{\Gamma(k)\Gamma(1-\frac{\alpha}{2})}n^{-1},&\alpha<2,\\
		\frac{l(\sqrt{a_n})}{a_n(k-1)},&\alpha=2,\E X_1^2=\infty,\\
		\frac{\alpha\Gamma(k-\frac{\alpha}{2})\Gamma(\frac{\alpha}{2})}{2\Gamma(k)}l(\sqrt{n})n^{-\frac{\alpha}{2}},&2\leq \alpha<2k,~\E X_1^2<\infty\\
		2k\widetilde{l}(\sqrt{n})n^{-k},&\alpha=2k,\E X_1^{2k}=\infty,\\
		\E X_1^{2k} n^{-k},&\alpha\geq 2k,~\E X_1^{2k}<\infty.
	\end{cases}
\end{align*}
\end{proof}

\begin{proposition}\label{prop:general_power}
Assume that $X_1,X_2,\cdots$ are i.i.d. regularly varying with $\alpha> 1$ and $\E X_1=0$. For positive integers $k_1,\cdots,k_r$, we denote $I_j=\{1\leq i\leq r:k_i=j\}$, $j=1,2,\cdots$.
\begin{itemize}
	\item If $1<\alpha\leq2$,
	\begin{align*}
		\E \left(Y_1^{k_1}\cdots Y_r^{k_r}\right)=O(n^{-r}).
	\end{align*}
	\item If $2<\alpha\leq3$, for any $\beta<\alpha$,
	\begin{align*}
		\E \left(Y_1^{k_1}\cdots Y_r^{k_r}\right)=O(n^{-\frac{\beta}{2}(r-|I_2|)-|I_2|}).
	\end{align*}
	\item If $\alpha>3$, for any $\beta<\alpha$,
	\begin{align*}
		\E \left(Y_1^{k_1}\cdots Y_r^{k_r}\right)=O(n^{-|I_1|-\frac{1}{2}\sum_{i=1}^{\lceil \alpha \rceil-1}i|I_i|-\frac{\beta}{2}(r-\sum_{i=1}^{\lceil \alpha \rceil-1}|I_i|)}).
	\end{align*}
\end{itemize}
\end{proposition}

\begin{proof}
For any $k_1,\cdots,k_r$, we have 
\begin{align*}
	\E \left(Y_1^{k_1}\cdots Y_r^{k_r}\right)=\frac{1}{\Gamma(\frac{k}{2})}\int_0^\infty s^{\frac{k}{2}-1}\prod_{i=1}^r\E \left( X_i^{k_i}e^{-sX_i^2}\right)\varphi^{n-r}(s)ds.
\end{align*}

Firstly, we consider the case  $1<\alpha\leq2$. When $k_i=1$, by the identity     
\begin{align*}
	sx^2 \E_\theta e^{- s \theta x^2}=1-e^{-sx^2}, \quad \forall sx^2>0,
\end{align*} 
where $\theta$ is uniformly distributed on $[0,1]$, we can get  
\begin{align*}
	X_i e^{-sX_i^2}=X_i \left(1-s X_i^2 \E_{\theta} e^{- s \theta X_i^2}\right),
\end{align*}
which yields
\begin{align*}
	\E\left(X_i e^{-sX_i^2}\right)=\E X_i-s \E_{X_i}\left\{\E_{\theta} \left(X_i^3 e^{- s \theta X_i^2}\right)\right\}.
\end{align*}
Thus, by Lemma \ref{lem:exact_limiting_phi},
\begin{align*}
	\left|\E \left(X_ie^{-sX_i^2}\right)\right|=&s\left|\E \left(X_i^3e^{-s\theta_i X_i^2}\right)\right|\leq s\E_{\theta_i} \left|\E_{X_i} \left(X_i^3e^{-s\theta_i X_i^2}\right)\right|\\
	\leq& s\E_{\theta_i}\left(\E_{X_i} \left(X_i^2e^{-s\theta_i X_i^2}\right)\cdot\E_{X_i} \left(X_i^4e^{-s\theta_i X_i^2}\right)\right)^{\frac{1}{2}}\\
	\lesssim& s\E_{\theta_i}\left((\theta_i s)^{\frac{\alpha}{2}-1}l\left(\frac{1}{\sqrt{\theta_i s}}\right)\cdot(\theta_i s)^{\frac{\alpha}{2}-2}l\left(\frac{1}{\sqrt{\theta_i s}}\right)\right)^{\frac{1}{2}}\\
	=&s^{\frac{\alpha-1}{2}}\E_{\theta_i}\left[\theta_i^{\frac{\alpha-3}{2}}l\left(\frac{1}{\sqrt{\theta_i s}}\right)\right],
\end{align*}
where $\theta_i$ is uniformly distributed on $[0,1]$ and independent of $X_i$.

Next when $k_i$ is odd and $k_i\geq 3$,
\begin{align*}
	\left|\E \left(X_i^{k_i}e^{-sX_i^2}\right)\right|\leq&\left(\E \left(X_i^{2}e^{-sX_i^2}\right)\cdot\E \left(X_i^{2k_i-2}e^{-sX_i^2}\right)\right)^{\frac{1}{2}}\\
	\lesssim& \left(s^{\frac{\alpha}{2}-1}l(\frac{1}{\sqrt{s}})\cdot s^{\frac{\alpha}{2}-k_i+1}l(\frac{1}{\sqrt{s}})\right)^{\frac{1}{2}}=s^{\frac{\alpha-k_i}{2}}l(\frac{1}{\sqrt{s}}).
\end{align*}
Combining the two pieces, by Lemma \ref{lem:norming_constants}, we can conclude that 
\begin{align*}
	&\left|\E Y_1^{k_1}\cdots Y_r^{k_r}\right|\\
	\lesssim&\int_0^\infty s^{\frac{k}{2}-1}\prod_{i\in I_1} \left[s^{\frac{\alpha-1}{2}}\E_{\theta_i}\left(\theta_i^{\frac{\alpha-3}{2}}l(\frac{1}{\sqrt{\theta_i s}})\right)\right]\prod_{1\leq i\leq r;i\not\in I_1}\left[s^{\frac{\alpha-k_i}{2}}l(\frac{1}{\sqrt{s}})\right]\varphi^{n-r}(s)ds\\
	\lesssim&\left[\prod_{i\in I_1}\E_{\theta_i}\theta_i^{\frac{\alpha-3}{2}}\right]\int_0^\infty s^{\frac{\alpha}{2}r-1}\left[l(\frac{1}{\sqrt{s}})\right]^r\varphi^{n-r}(s)ds\\
	\lesssim&\int_0^{\infty}\left(\frac{t}{a_n}\right)^{\frac{\alpha}{2}r-1}\left[l(\sqrt{a_n})\right]^r\varphi^{n-r}(\frac{t}{a_n})\frac{1}{a_n}dt\lesssim n^{-r}.
\end{align*}

Secondly, we consider the case $2<\alpha\leq 3$. For $i\in I_1$, we choose $\beta\in(2,\alpha)$ arbitrarily. Since $\E|X_i|^\beta<\infty$ and 
\begin{align}
	f(x)=x^{3-\beta}e^{-s\theta x^2}\leq \left(\frac{3-\beta}{2s\theta}\right)^{\frac{3-\beta}{2}}e^{-\frac{3-\beta}{2}},\label{form:upper_bound_trick}
\end{align}
we have 
\begin{align*}
	\left|\E \left(X_ie^{-sX_i^2}\right)\right|=&s\left|\E \left(X_i^3e^{-s\theta_i X_i^2}\right)\right|\leq s\E \left(|X_i|^\beta\cdot |X_i|^{3-\beta}e^{-s\theta_i X_i^2}\right)\lesssim \E\theta_i^{\frac{\beta-3}{2}}s^{\frac{\beta-1}{2}}\lesssim s^{\frac{\beta-1}{2}}.
\end{align*}
For $i\in I_2$,
\begin{align*}
	\left|\E \left(X_i^2e^{-sX_i^2}\right)\right|\lesssim\E|X_i|^2\lesssim 1.
\end{align*}
For $i\in I_3$,
\begin{align*}
	\left|\E \left(X_i^{k_i}e^{-sX_i^2}\right)\right|\leq\E\left(|X_i|^\beta\cdot|X_i|^{k_i-\beta}e^{-sX_i^2}\right)\lesssim s^{-\frac{k_i-\beta}{2}}.
\end{align*}
So we have 
\begin{align*}
	\left|\E \left(Y_1^{k_1}\cdots Y_r^{k_r}\right)\right|\lesssim&\int_0^\infty s^{\frac{k}{2}-1}\cdot s^{\frac{\beta-1}{2}|I_1|}\prod_{i\in I_3}s^{-\frac{k_i-\beta}{2}}\varphi^{n-r}(s)ds\\
	=&\int_0^\infty s^{\frac{\beta}{2}(r-|I_2|)+|I_2|-1}\varphi^{n-r}(s)ds\\
	=&\int_0^\infty (\frac{t}{n})^{\frac{\beta}{2}(r-|I_2|)+|I_2|-1}\varphi^{n-r}(\frac{t}{n})\frac{1}{n}dt\\
	\lesssim& n^{-\frac{\beta}{2}(r-|I_2|)-|I_2|}.
\end{align*}
Finally, we deal with the case $\alpha>3$. For convenience, we denote $H=\cup_{2\leq j\leq\lceil\alpha\rceil-1} I_j$ and $\Lambda=\cup_{j\geq\lceil\alpha\rceil}I_j$. For $i\in I_1$,
\begin{align*}
	\left|\E \left(X_ie^{-sX_i^2}\right)\right|=s\left|\E \left(X_i^3e^{-s\theta_i X_i^2}\right)\right|\lesssim s\E|X_i|^3\lesssim s.
\end{align*}
For $i\in H$, 
\begin{align*}
	\left|\E \left(X_i^{k_i}e^{-sX_i^2}\right)\right|\lesssim\E|X_i|^{k_i}\lesssim 1.
\end{align*}
For $i\in \Lambda$, with similar process in \eqref{form:upper_bound_trick},
\begin{align*}
	\left|\E \left(X_i^{k_i}e^{-sX_i^2}\right)\right|\leq\E\left(|X_i|^\beta\cdot|X_i|^{k_i-\beta}e^{-sX_i^2}\right)\lesssim s^{-\frac{k_i-\beta}{2}}.
\end{align*}
So we have 
\begin{align*}
	\left|\E \left(Y_1^{k_1}\cdots Y_r^{k_r}\right)\right|\lesssim&\int_0^\infty s^{\frac{k}{2}-1}\cdot s^{|I_1|}\prod_{i\in \Lambda}s^{-\frac{k_i-\beta}{2}}\varphi^{n-r}(s)ds\\
	=&\int_0^\infty s^{|I_1|+\frac{1}{2}\sum_{i\in H}k_i+\frac{\beta}{2}|\Lambda|-1}\varphi^{n-r}(s)ds\\
	=&\int_0^\infty (\frac{t}{n})^{|I_1|+\frac{1}{2}\sum_{i\in H}k_i+\frac{\beta}{2}|\Lambda|-1}\varphi^{n-r}(\frac{t}{n})\frac{1}{n}dt\\
	\lesssim& n^{-|I_1|-\frac{1}{2}\sum_{i\in H}k_i-\frac{\beta}{2}|\Lambda|}.
\end{align*}
\end{proof}

\section{Proofs for symmetric distributions}
The proofs are mainly based on the asymptotic results of the quadratic forms $(\Y_1\trans\Y_2)^2$. Recall the definition 
\begin{align*}
\X_i=\begin{pmatrix}
	X_{i1}\\
	\vdots\\
	X_{in}
\end{pmatrix}, \quad i=1,\ldots,p,
\end{align*}
and 
\begin{align*}
\Y_i=\frac{\X_i}{\|\X_i\|}, i=1,\ldots,p,
\end{align*}
where $\{X_{ij},i=1,\ldots,p,j=1,\ldots,n\}$ are i.i.d. from the population $X$. 
\begin{proposition} \label{all_moments_symmetric}
Assuming $X \ind -X$, we have
\begin{align*}
	\E(\Y_1\trans \Y_2)^2=\frac{1}{n}, \quad  \cov\left((\Y_1\trans \Y_2)^2,(\Y_1\trans \Y_3)^2\right)=0,
\end{align*}
and
\begin{align*}
	\var\left((\Y_1\trans \Y_2)^2\right) =\frac{n(n+2)}{n-1} \left(\E Y_1^4\right)^2-\frac{6}{n-1}\E Y_1^4+\frac{2n+1}{n^2(n-1)}.
\end{align*}  
\end{proposition}
\begin{proof}
Noting 
\begin{align}
	\cov(\Y)=&\E \Y\Y\trans=\frac{1}{n}\bI_n,\label{form:cov_Y}
\end{align}
we can get 
\begin{align*}
	\E(\Y_1\trans \Y_2)^2=&\E \tr(\Y_1\Y_1\trans\Y_2\Y_2\trans)=\tr\left( \frac{1}{n^2}\bI_n\right)=\frac{1}{n}.
\end{align*}   
For the covariance, we have
\begin{align*}
	\cov\left((\Y_1\trans \Y_2)^2,(\Y_1\trans \Y_3)^2\right)=\var\left( \Y_1\trans\left(\frac{1}{n}\bI_n\right)\Y_1 \right)=0,
\end{align*}
and 
\begin{align*}
	&\var \left(\left(\Y_1 \trans\Y_2 \right)^2 \right)=\cov\left(\sum_{i,j=1}^n Y_{1i}Y_{1j}Y_{2i}Y_{2j},\sum_{{i'},{j'}=1}^n Y_{1{i'}}Y_{1{j'}}Y_{2{i'}}Y_{2{j'}}  \right)\\
	=&n \cov\left(Y^2_{11} Y^2_{21}, \sum_{i,j=1}^n Y_{1i}Y_{1j}Y_{2i}Y_{2j}\right)+n(n-1)\cov\left(Y_{11}Y_{12} Y_{21}Y_{22}, \sum_{i,j=1}^n Y_{1i}Y_{1j}Y_{2i}Y_{2j}\right)\\
	=& n \var \left(Y^2_{11} Y^2_{21}\right)+n(n-1)  \cov\left(Y^2_{11} Y^2_{21},Y^2_{12}Y^2_{22}\right) +2n(n-1)\var\left(Y_{11}Y_{12} Y_{21}Y_{22}\right)\\
	=&n \left(\E Y_1^4\right)^2-\frac{1}{n^3}+n(n-1)\left(\E Y_1^2 Y_2^2\right)^2-\frac{n-1}{n^3}+2n(n-1)\left(\E Y_1^2 Y_2^2\right)^2\\
	=&\frac{n(n+2)}{n-1} \left(\E Y_1^4\right)^2-\frac{6}{n-1}\E Y_1^4+\frac{2n+1}{n^2(n-1)},
\end{align*}
where the last equality follows from the fact that
\begin{align}
	1=\E (Y_1^2+\cdots+Y_n^2)^2=n \E Y_1^4+n(n-1)\E Y_1^2 Y_2^2.\label{form:identity_for_one}
\end{align} 
\end{proof}

\begin{proposition} \label{quadratic_form_symmetric}
Assuming $X \ind -X$, for any deterministic and symmetric matrices $\A=(a_{ij})_{n\times n}$ and $\B=(b_{ij})_{n\times n}$, we have
\begin{align*}
	&\cov\left(\Y\trans \A \Y, \Y\trans \B \Y\right)\\
	=&\frac{n(n+2)\E Y_1^4-3}{n(n-1)}\tr\left(\A \circ \B\right)+\frac{1-n^2 \E Y_1^4}{n^2(n-1)}\tr(\A)\tr(\B)+\frac{2-2n \E Y_1^4}{n(n-1)} \tr(\A\trans \B).
\end{align*}
\end{proposition}
\begin{proof}
A direct calculation gives
\begin{align*}
	&\cov\left(\Y\trans \A \Y, \Y\trans \B \Y\right)=\cov\left(\sum_{i,j=1}^n a_{ij}Y_{i}Y_{j},\sum_{{i'},{j'}=1}^n b_{i'j'} Y_{i'}Y_{j'}  \right)\\
	=&\cov\left(\sum_{i=1}^n a_{ii}Y^2_{i},\sum_{j=1}^n b_{jj}Y^2_j \right)+4\sum_{i<j}a_{ij}b_{ij}\var(Y_iY_j)\\
	=& \var(Y_1^2)\sum_{i=1}^n a_{ii}b_{ii}+\cov\left(Y_1^2,Y_2^2\right)\sum_{i \neq j}a_{ii}b_{jj}+4\var(Y_1Y_2)\sum_{i<j}a_{ij}b_{ij}\\
	=& \left(\var(Y_1^2)-\cov\left(Y_1^2,Y_2^2\right)-2\var(Y_1Y_2)\right)\sum_{i=1}^n a_{ii}b_{ii}+\cov\left(Y_1^2,Y_2^2\right)\sum_{i,j}a_{ii}b_{jj}\\
	&+2\var(Y_1Y_2)\sum_{i,j}a_{ij}b_{ji}\\
	=&\left(\E Y_1^4-3\E Y_1^2Y_2^2\right)\tr\left(\A \circ \B\right)+\left(\E Y_1^2Y_2^2-\frac{1}{n^2}\right)\tr(\A)\tr(\B)+2\E Y_1^2Y_2^2 \tr(\A\trans \B).
\end{align*}
This, together with \eqref{form:identity_for_one}, completes the proof.
\end{proof}

\begin{proof}[Proof of Proposition \ref{thm:expectation_and_variance}]
By Proposition \ref{all_moments_symmetric}, we have
\begin{align*}
	\E \tr(\R^2)=p+p(p-1)\E\left(\Y_1\trans\Y_2\right)^2=p+\frac{p(p-1)}{n}
\end{align*}
and
\begin{align*}
	&\var(\tr(\R^2))=\cov \left(\sum_{i \neq j}\left(\Y_i\trans\Y_j\right)^2, \sum_{{i'} \neq {j'}}\left(\Y_{i'}\trans\Y_{j'}\right)^2 \right)\\
	=& p(p-1) \cov \left(\left(\Y_1 \trans\Y_2 \right)^2,  \sum_{i \neq j}\left(\Y_i\trans\Y_j\right)^2 \right)\\
	=& 2 p(p-1) \var \left(\left(\Y_1 \trans\Y_2 \right)^2 \right)+4 p(p-1)(p-2) \cov \left(\left(\Y_1 \trans\Y_2 \right)^2,  \left(\Y_1\trans\Y_3\right)^2 \right)\\
	=&2p(p-1) \left[\frac{n(n+2)}{n-1} \left(\E Y_{11}^4\right)^2-\frac{6}{n-1}\E Y_{11}^4+\frac{2n+1}{n^2(n-1)} \right].
\end{align*}
\end{proof}

\begin{proof}[Proof of Theorem \ref{thm:CLT}]     
Let $V_n=\var(T)$. Define the martingale difference as follows
\begin{align*}
	M_{n,k}=&\frac{1}{\sqrt{V_n}}\left(\E[T|\Y_1,\cdots,\Y_k]-\E[T|\Y_1,\cdots,\Y_{k-1}]\right)\\
	=&\frac{2}{\sqrt{V_n}}\tr \left(  \sum_{i<k} \Y_i\Y_i\trans+\frac{p-k}{n}\bI_n\right)\left(  \Y_k \Y_k\trans-\frac{1}{n}\bI_n\right)\\
	=&\frac{2}{\sqrt{V_n}}\left(\sum_{i<k}\left(\Y_i\trans\Y_k\right)^2-\frac{k-1}{n}\right)\\
	=&\frac{2}{\sqrt{V_n}} \sum_{i<k} \Y_k \trans  \left(\Y_i \Y_i \trans-\frac{1}{n}\bI_n\right)\Y_k,
\end{align*}
where we use the facts that $\Y_i\trans\Y_i=1$, \eqref{form:cov_Y} and that
\begin{align*}
	T=&\sum_{i,j=1}^p \left(\Y_i\trans\Y_j\right)^2=\tr \left(  \sum_{i=1}^p \Y_i\Y_i\trans \right)\left(  \sum_{j=1}^p \Y_j\Y_j\trans \right)\\
	=&\tr \left(  \sum_{i<k} \Y_i\Y_i\trans+\Y_k \Y_k\trans+\sum_{i>k}\Y_i\Y_i\trans \right)\left(  \sum_{j<k} \Y_j\Y_j\trans+\Y_k \Y_k\trans+\sum_{j>k}\Y_j\Y_j\trans \right)\\
	=&\tr \left(  \sum_{i \neq k} \Y_i\Y_i\trans \right)\left(  \sum_{j \neq k} \Y_j\Y_j\trans \right)+1+2 \tr \left(  \sum_{i<k} \Y_i\Y_i\trans+\sum_{i>k}\Y_i\Y_i\trans \right)\left(  \Y_k \Y_k\trans\right).
\end{align*}
To apply the martingale central limit theorem, we need to verify that
\begin{itemize}
	\item[(i).] $\sum_{k=1}^p\E M_{n,k}^2\to1$.
	\item[(ii).] $\var \left( \sum_{k=1}^p\E[M_{n,k}^2|\Y_1,\cdots,\Y_{k-1}]\right)\to 0$.
	\item[(iii).] $\sum_{k=1}^p\E M_{n,k}^4\to 0$.
\end{itemize}

First, (i) is obvious since 
\begin{align*}
	\sum_{k=1}^p \E M_{n,k}^2= \sum_{k=1}^p \var\left(M_{n,k}\right)=\var\left(\sum_{k=1}^p M_{n,k} \right)=\frac{\var(T)}{V_n}=1. 
\end{align*}

For (ii), note that
\begin{align*}
	M_{n,k}= \frac{2}{\sqrt{V_n}} \Y_k \trans \left( \sum_{i=1}^{k-1}   \left(\Y_i \Y_i \trans-\frac{1}{n}\bI_n\right) \right) \Y_k,
\end{align*}
and that
\begin{align*}
	\E[M_{n,k}|\Y_1,\cdots,\Y_{k-1}]=\frac{2}{n \sqrt{V_n}}  \tr\left( \sum_{i=1}^{k-1}   \left(\Y_i \Y_i \trans-\frac{1}{n}\bI_n\right) \right)=0.
\end{align*}
Proposition \ref{quadratic_form_symmetric} implies that
\begin{align*}
	&\frac{V_n}{4}\E[M_{n,k}^2|\Y_1,\cdots,\Y_{k-1}]\\
	=&\left(\E Y_1^4-3\E \left(Y_1^2Y_2^2\right)\right)\sum_{i,j=1}^{k-1}\tr \left(  \left(\Y_i \Y_i \trans-\frac{1}{n}\bI_n\right)  \circ \left(\Y_j \Y_j \trans-\frac{1}{n}\bI_n\right) \right)\\
	&+2\E \left(Y_1^2Y_2^2\right)\sum_{i,j=1}^{k-1}\tr \left(  \left(\Y_i \Y_i \trans-\frac{1}{n}\bI_n\right) \left(\Y_j \Y_j \trans-\frac{1}{n}\bI_n\right) \right) \\
	=&\left(\E Y_1^4-3\E \left(Y_1^2Y_2^2\right)\right)\sum_{i,j=1}^{k-1} \left(\left\|\Y_i \circ \Y_j\right\|^2_2-\frac{1}{n}\right)+2\E \left(Y_1^2Y_2^2\right)\sum_{i,j=1}^{k-1} \left(\left( \Y_i \trans \Y_j\right)^2-\frac{1}{n}\right).
\end{align*}
It follows that
\begin{align}
	&\sum_{k=1}^p\E[M_{n,k}^2|\Y_1,\cdots,\Y_{k-1}]=\frac{4\left(\E Y_1^4-3\E \left(Y_1^2Y_2^2\right)\right)}{V_n}\sum_{i=1}^{p}\left(p-i\right) \left(\left\|\Y_i \circ \Y_i\right\|^2_2-\frac{1}{n}\right)\notag\\
	&+\frac{8\left(\E Y_1^4-3\E \left(Y_1^2Y_2^2\right)\right)}{V_n}\sum_{i<j}^{p}\left(p-j\right) \left(\left\|\Y_i \circ \Y_j\right\|^2_2-\frac{1}{n}\right)\notag\\
	&+\frac{16 \E \left(Y_1^2Y_2^2\right)}{V_n}\sum_{i<j} \left(p-j\right) \left(\left( \Y_i \trans \Y_j\right)^2-\frac{1}{n}\right)+\frac{4 \E \left(Y_1^2Y_2^2\right)}{V_n}p(p-1)\left(1-\frac{1}{n}\right).\label{form:sum_of_martingale_difference}
\end{align}
It is  straightforward to verify that 
\begin{gather*}
	\cov\left(\left\|\Y_1 \circ \Y_2\right\|^2_2,\left\|\Y_1 \circ \Y_3\right\|^2_2\right)=0, \quad \cov\left(\left( \Y_1 \trans \Y_2\right)^2,\left( \Y_1 \trans \Y_3\right)^2\right)=0.
\end{gather*}
This ensures that
\begin{align}
	\var\left( \sum_{i=1}^{p}\left(p-i\right) \left(\left\|\Y_i \circ \Y_i\right\|^2_2 \right) \right)=&\var\left(\left\|\Y_1 \circ \Y_1\right\|^2_2  \right) \sum_{i=1}^{p-1} i^2\notag\\
	=&\var\left(\left\|\Y_1 \circ \Y_1\right\|^2_2  \right) \frac{p(p-1)(2p-1)}{6},\notag\\
	\var\left( \sum_{i<j}^{p}\left(p-j\right) \left(\left\|\Y_i \circ \Y_j\right\|^2_2-\frac{1}{n}\right) \right)=&\var \left(\left\|\Y_1 \circ \Y_2\right\|^2_2 \right) \sum_{i<j}\left(p-j\right)^2 \notag\\
	\leq & \frac{p^2(p-1)^2}{3}\var \left(\left\|\Y_1 \circ \Y_2\right\|^2_2 \right),\label{form:var_bound_for_hadamard}\\
	\var\left( \sum_{i<j} \left(p-j\right) \left(\left( \Y_i \trans \Y_j\right)^2-\frac{1}{n}\right) \right)=&\var \left(\left( \Y_1 \trans \Y_2\right)^2\right) \sum_{i<j}\left(p-j\right)^2  \notag\\
	\leq & \frac{p^2(p-1)^2}{3} \var \left(\left( \Y_1 \trans \Y_2\right)^2 \right).\label{form:var_bound_for_matrixProduct}
\end{align}
Furthermore, note that
\begin{align*}
	&  \var\left(\left\|\Y_1 \circ \Y_1\right\|^2_2  \right)=n \var\left(Y^4_{11} \right)+n(n-1) \cov\left(Y^4_{11},Y^4_{12}\right)\\
	&= n\E Y_1^8+n(n-1)\E Y_1^4Y_2^4-n^2\left(\E Y_1^4\right)^2,\\
	&  \var\left(\left\|\Y_1 \circ \Y_2\right\|^2_2  \right)=\cov\left(\sum_{i=1}^n Y^2_{1i}Y^2_{2i},\sum_{j=1}^n Y^2_{1j}Y^2_{2j}  \right)=n \cov\left(Y^2_{11} Y^2_{21}, \sum_{j=1}^n Y^2_{1j}Y^2_{2j}\right)\\
	&= n \var \left(Y^2_{11} Y^2_{21}\right)+n(n-1)  \cov\left(Y^2_{11} Y^2_{21},Y^2_{12}Y^2_{22}\right)\leq \var \left(\left( \Y_1 \trans \Y_2\right)^2 \right).
\end{align*}
A direct calculation shows that
\begin{align}
	V_n=2 p(p-1) \var \left(\left(\Y_1 \trans\Y_2 \right)^2 \right)=&2p(p-1) \left[\frac{n(n+2)}{n-1} \left(\E Y^4\right)^2-\frac{6}{n-1}\E Y^4+\frac{2n+1}{n^2(n-1)} \right]\notag\\
	=&2 p(p-1) \left(n \left(\E Y^4\right)^2-\frac{6}{n}\E Y^4+\frac{2}{n^2}\right)(1+o(1)).\label{form:V_n_expression}
\end{align}
We conclude from \eqref{form:var_bound_for_hadamard}, \eqref{form:var_bound_for_matrixProduct} and \eqref{form:V_n_expression} that
\begin{gather*}
	\frac{\left(\E Y^4\right)^2}{V_n^2}  \var\left( \sum_{i=1}^{p}\left(p-i\right) \left(\left\|\Y_i \circ \Y_i\right\|^2_2 \right) \right)\leq \frac{\left(\E Y_1^4\right)^2\left(n\E Y_1^8+n(n-1)\E Y_1^4Y_2^4\right)}{12(p-1) \left(\var \left(\left(\Y_1 \trans\Y_2 \right)^2 \right)\right)^2},\\
	\frac{\left(\E Y^4\right)^2}{V_n^2}\var\left( \sum_{i<j}^{p}\left(p-j\right) \left(\left\|\Y_i \circ \Y_j\right\|^2_2-\frac{1}{n}\right) \right) \leq \frac{\left(\E Y_1^4\right)^2}{12 \var \left(\left(\Y_1 \trans\Y_2 \right)^2 \right)}\to 0,\\
	\frac{\left(\E Y_1^2 Y_2^2\right)^2}{V_n^2}\var\left( \sum_{i<j} \left(p-j\right) \left(\left( \Y_i \trans \Y_j\right)^2-\frac{1}{n}\right) \right)\leq \frac{\left(\E Y_1^4\right)^2}{12 \var \left(\left(\Y_1 \trans\Y_2 \right)^2 \right)}\to 0.
\end{gather*}
This, together with \eqref{form:sum_of_martingale_difference}, implies that
\begin{align} \label{condition1}
	\var \left( \sum_{k=1}^p\E[M_{n,k}^2|\Y_1,\cdots,\Y_{k-1}]\right)\leq C \frac{\left(\E Y_1^4\right)^2\left(n\E Y_1^8+n(n-1)\E Y_1^4Y_2^4\right)}{p \left(n \left(\E Y^4\right)^2+\frac{2}{n^2}\right)^2}+o(1).
\end{align}

Finally, we verify Lindeberg's condition that 
\begin{align*}
	\sum_{k=1}^p\E M_{n,k}^4 \to 0.
\end{align*}
Since
\begin{align*}
	M_{n,k}= \frac{2}{\sqrt{V_n}} \Y_k \trans \left( \sum_{i=1}^{k-1}   \left(\Y_i \Y_i \trans-\frac{1}{n}\bI_n\right) \right) \Y_k=\frac{2}{\sqrt{V_n}} \sum_{i=1}^{k-1} \Y_i \trans  \left(\Y_k \Y_k \trans-\frac{1}{n}\bI_n\right)\Y_i,
\end{align*}
we have
\begin{align*}
	M_{n,k}^4=&\frac{16}{V_n^2} \sum_{i_1,i_2,i_3,i_4=1}^{k-1} \left(\Y_{i_1} \trans \A_k \Y_{i_1}\right) \left(\Y_{i_2} \trans \A_k \Y_{i_2}\right)  \left(\Y_{i_3} \trans \A_k \Y_{i_3}\right) \left( \Y_{i_4} \trans \A_k \Y_{i_4}\right),
\end{align*}
where $\A_k=\Y_k \Y_k \trans-\frac{1}{n}\bI_n$. Noting
\begin{align*}
	\E \left( \Y_i \trans  \A_k \Y_i \mid \Y_k \right)=\frac{1}{n}\tr(\A_k)=0,~i=1,\cdots,k-1,
\end{align*}
thus we can get 
\begin{align}
	\E M_{n,k}^4=& \frac{16}{V_n^2} \sum_{i=1}^{k-1} \E \left(\Y_{i} \trans \A_k \Y_{i}\right)^4+ \frac{48}{V_n^2} \sum_{i \neq j}^{k-1} \E \left(\Y_{i} \trans \A_k \Y_{i}\right)^2\left(\Y_{j} \trans \A_k \Y_{j}\right)^2 \notag\\
	=& \frac{16(k-1)}{V_n^2} \E \left(\Y_{1} \trans \A_2 \Y_{1}\right)^4+ \frac{48(k-1)(k-2)}{V_n^2} \E \left(\Y_{1} \trans \A_3 \Y_{1}\right)^2\left(\Y_{2} \trans \A_3 \Y_{2}\right)^2. \label{form:E_M^4}
\end{align}
For the first term above, we have
\begin{align*}  
	\E \left(\Y_{1} \trans \A_2 \Y_{1}\right)^4=\E \left( \left(\Y_1\trans \Y_2 \right)^2-\frac{1}{n}  \right)^4\leq 8 \E \left(\left(\Y_1\trans \Y_2 \right)^8+\frac{1}{n^4}\right),
\end{align*}
and 
\begin{align*}
	\E\left(\Y_1\trans\Y_2\right)^8=\E \left(\sum_{i=1}^n Y_{1i}Y_{2i}\right)^8=\sum_{i_1,\cdots,i_8}\E \left(Y_{1i_1}\cdots Y_{1i_8}\right)\E \left(Y_{2i_1}\cdots Y_{2i_8}\right)=\sum_{i_1,\cdots,i_8} \left(\E \left(Y_{i_1}\cdots Y_{i_8}\right)\right)^2.
\end{align*}
By the symmetry of $\{Y_i\}_{i=1}^n$, we can get
\begin{align}
	\E\left(\Y_1\trans\Y_2\right)^8=&\sum_{|\{i_1\}\cup\cdots \cup \{i_8\}|=1} \left(\E \left(Y_{i_1}\cdots Y_{i_8}\right)\right)^2+ \sum_{|\{i_1\}\cup\cdots \cup \{i_8\}|=2} \left(\E \left(Y_{i_1}\cdots Y_{i_8}\right)\right)^2\notag\\
	&+\sum_{|\{i_1\}\cup\cdots \cup \{i_8\}|=3} \left(\E \left(Y_{i_1}\cdots Y_{i_8}\right)\right)^2+\sum_{|\{i_1\}\cup\cdots \cup \{i_8\}|\geq 4} \left(\E \left(Y_{i_1}\cdots Y_{i_8}\right)\right)^2 \notag\\
	=&n\left(\E Y_{1}^8\right)^2+n(n-1) \sum_{k_1+k_2=4,k_1,k_2\geq 1} \frac{8!}{(2k_1)!(2k_2)!}\left(\E \left(Y_{1}^{2k_1}Y_{2}^{2k_2}\right)\right)^2\notag\\
	&+n(n-1)(n-2) \sum_{k_1+k_2+k_3=4,k_1,k_2,k_3\geq 1} \frac{8!}{(2k_1)!(2k_2)!(2k_3)!}\left(\E \left(Y_{1}^{2k_1}Y_{2}^{2k_2}Y_{3}^{2k_3}\right)\right)^2+O(n^{-4})\notag\\
	\leq & C \left( n\left(\E Y_{1}^8\right)^2+n^2 \left(\E \left(Y_{1}^6Y_2^2\right)\right)^2+n^2 \left(\E \left(Y_{1}^4 Y_2^4\right)\right)^2+n^3 \left(\E \left(Y_{1}^2 Y_2^2 Y_3^4 \right)\right)^2+n^{-4}\right).\label{form:E8}
\end{align}
For the second term of \eqref{form:E_M^4}, by Proposition \ref{quadratic_form_symmetric} and Proposition \ref{prop:even_power},
\begin{align}
	\E\left(\Y_1\trans\A_3\Y_1\right)^2\left(\Y_2\trans\A_3\Y_2\right)^2=&\E\left[\frac{n(n+2)\E Y^4-3}{n(n-1)}\tr\left(\A_3 \circ \A_3 \right)+\frac{2-2n \E Y^4}{n(n-1)} \tr(\A_3 \trans \A_3)\right]^2\nonumber \\
	=&\E\left[\frac{n(n+2)\E Y^4-3}{n(n-1)}\sum_{i=1}^n\left(Y_{i}^2-\frac{1}{n}\right)^2+\frac{2-2n \E Y^4}{n^2}\right]^2 \nonumber \\
	\leq & C \left(\E Y^4\right)^2 \left( n\E Y_1^8+n(n-1)\E Y_1^4Y_2^4\right)+\frac{C}{n^4},\label{form:E2_2}
\end{align}
where we use the trivial bound $ 1\leq n^2 \E Y^4 \leq n$ by \eqref{form:identity_for_one}. Combining \eqref{form:V_n_expression}, \eqref{form:E_M^4}, \eqref{form:E8} and \eqref{form:E2_2}, we can get
\begin{align}
	\sum_{k=1}^p \E M_{n,k}^4 \leq & C_1  \frac{n\left(\E Y_{1}^8\right)^2+n^2 \left(\E Y_{1}^6Y_2^2\right)^2+n^2 \left(\E Y_{1}^4 Y_2^4\right)^2+n^3 \left(\E Y_{1}^2 Y_2^2 Y_3^4 \right)^2}{p^2\left(n \left(\E Y^4\right)^2+\frac{2}{n^2}\right)^2}\nonumber  \\
	& + C_2  \frac{\left(\E Y^4\right)^2 \left( n\E Y_1^8+n(n-1)\E Y_1^4Y_2^4\right)}{p\left(n \left(\E Y^4\right)^2+\frac{2}{n^2}\right)^2}+\frac{C_3}{p}.  \label{condition2}
\end{align}

From \eqref{condition1} and \eqref{condition2}, we can conclude that (ii) and (iii) hold if 
\begin{align}
	&\frac{n\left(\E Y_{1}^8\right)^2+n^2 \left(\E Y_{1}^6Y_2^2\right)^2+n^2 \left(\E Y_{1}^4 Y_2^4\right)^2+n^3 \left(\E Y_{1}^2 Y_2^2 Y_3^4 \right)^2}{p^2\left(n \left(\E Y^4\right)^2+\frac{2}{n^2}\right)^2} \to  0; \label{clt1} \\
	& \frac{\left(\E Y^4\right)^2 \left( n\E Y_1^8+n(n-1)\E Y_1^4Y_2^4\right)}{p\left(n \left(\E Y^4\right)^2+\frac{2}{n^2}\right)^2} \to  0. \label{clt2}
\end{align}
By exploring the trivial identities,
\begin{align*}
	& \E Y_1^4=\E \left[Y_1^4\left(Y_1^2+\cdots+Y_n^2 \right)\right]=\E Y_1^6+(n-1) \E \left(Y_1^4Y_2^2\right);\\
	& \E Y_1^6=\E \left[Y_1^6\left(Y_1^2+\cdots+Y_n^2 \right)\right]=\E Y_1^8+(n-1) \E \left(Y_1^6Y_2^2\right);\\
	& \E \left(Y_1^4Y_2^2\right)=\E \left[Y_1^4Y_2^2\left(Y_1^2+\cdots+Y_n^2 \right)\right]=\E  \left(Y_1^6Y_2^2\right)+\E \left(Y_1^4Y_2^4\right)+(n-2) \E \left(Y_1^2Y_2^2Y_3^4\right);\\
	& 2\left(\E\left(Y_1^6Y_2^2\right)-\E\left(Y_1^4Y_2^4\right)\right)=\E \left[Y_1^2Y_2^2\left(Y_1^2-Y_2^2\right)^2\right],
\end{align*}
we can get the bounds 
\begin{align*}
	\E \left(Y_1^2Y_2^2Y_3^4\right)\leq \frac{1}{n-2} \E \left(Y_1^4Y_2^2\right)\leq \frac{1}{(n-1)(n-2)}  \E Y_1^4,\quad \E \left(Y_1^4Y_2^4\right) \leq \E \left(Y_1^6Y_2^2\right)\leq \frac{1}{n-1} \E Y_1^6,
\end{align*}
with which, the condition \eqref{clt1} is reduced to 
\begin{align*}
	\frac{n\left(\E Y_{1}^8\right)^2+ 2 \left(\E Y_{1}^6\right)^2+n^{-1} \left(\E Y_{1}^4 \right)^2}{p^2\left(n \left(\E Y^4\right)^2+\frac{2}{n^2}\right)^2} \leq \frac{\left(n \E Y_{1}^8+\E Y_{1}^4\right)^2}{n p^2\left(n \left(\E Y^4\right)^2+\frac{2}{n^2}\right)^2}=\left(\frac{n \E Y_{1}^8+\E Y_{1}^4}{\sqrt{n} p \left(n \left(\E Y^4\right)^2+\frac{2}{n^2}\right)} \right)^2.
\end{align*}
For condition \eqref{clt2},
\begin{align*}
	\frac{\left(\E Y^4\right)^2 \left( n\E Y_1^8+n(n-1)\E Y_1^4Y_2^4\right)}{p\left(n \left(\E Y^4\right)^2+\frac{2}{n^2}\right)^2} \leq & \frac{\E Y_1^8+(n-1)\E Y_1^4Y_2^4 }{p \left(n \left(\E Y^4\right)^2+\frac{2}{n^2}\right)} \leq \frac{\E Y_1^6 }{p \left(n \left(\E Y^4\right)^2+\frac{2}{n^2}\right)}\\
	\leq & \frac{\left(\E Y_1^4 n \E Y_1^8  \right)^{1/2} }{ \sqrt{n} p \left(n \left(\E Y^4\right)^2+\frac{2}{n^2}\right)} \leq  \frac{n \E Y_{1}^8+\E Y_{1}^4}{\sqrt{n} p \left(n \left(\E Y^4\right)^2+\frac{2}{n^2}\right)}.
\end{align*}
Therefore, it is sufficient to control
\begin{align} \label{clt0-0}
	\frac{n \E Y_{1}^8+\E Y_{1}^4}{\sqrt{n} p \left(n \left(\E Y_1^4\right)^2+\frac{2}{n^2}\right)} \to 0. 
\end{align}  
By Proposition \ref{prop:even_power}, 
\begin{align*}
	\frac{n \E Y_{1}^8+\E Y_{1}^4}{\sqrt{n} p \left(n \left(\E Y_1^4\right)^2+\frac{2}{n^2}\right)} \sim \begin{cases}
		\frac{\E Y_1^8}{\sqrt{n}p\left(\E Y_1^4\right)^2},&\alpha<3,\\
		\frac{n^3\E Y_1^8}{\sqrt{n}p\left(n^3\left(\E Y_1^4\right)^2+2\right)}, &\alpha=3,\\
		\frac{n^3\E Y_1^8}{2\sqrt{n}p}, &3<\alpha<6,\\
		\frac{n^3\E Y_1^8+n^2\E Y_1^4}{2\sqrt{n}p}, &\alpha=6,\\
		\frac{n^2\E Y_1^4}{2\sqrt{n}p}, &\alpha>6,
	\end{cases}
\end{align*}
which is ensured by
\begin{align*}
	\begin{cases}
		\frac{n}{p^2}\to 0,&\alpha<2,\\
		\frac{n^{\alpha-1}}{p^2} \frac{\tilde{l}^2(\sqrt{a_n})}{l^2(\sqrt{a_n})}  \to 0, &\alpha=2,\E X^{2}=\infty,\\
		\frac{n^{\alpha-1}}{p^2} \frac{1}{l^2(\sqrt{n})}  \to 0,&2\leq\alpha<3, \E X^{2}<\infty\\
		\frac{n}{p}\frac{l(\sqrt{n})}{l^2(\sqrt{n})+2}\to 0,&\alpha=3,\\
		\frac{n^{5-\alpha}l^2(\sqrt{n})}{p^2}\to 0,&3<\alpha\leq 5,\\
		n, p \to \infty,&\alpha>5.
	\end{cases}
\end{align*}
\end{proof}

\section{Proof for general distributions}
\begin{lemma}[\citealt{bingham1981conditions}]\label{lem:extreme_value_property}
Assume that $X_1,X_2,\cdots$ are i.i.d. non-negative and regularly varying with $0<\alpha\leq 1$. Consider 
\begin{align*}
	L_n=\frac{X_1+\cdots+X_n}{\max(X_1,\cdots,X_n)}.
\end{align*}
Then for any $\epsilon>0$, we have 
\begin{align*}
	\E L_n=\begin{cases}
		(1-\alpha)^{-1}+o(1),& 0<\alpha<1,\\
		o(n^{\epsilon}), & \alpha=1.
	\end{cases}
\end{align*}
\end{lemma}

\begin{lemma}\label{lem:bounds_for_Y_bar}
Assume that $X_1,X_2,\cdots$ are i.i.d. regularly varying with $\alpha>1$ and $\E X_1=0$. For any $k,l\geq 1$,
\begin{align*}
	\E\bar{Y}^l=O(n^{-l}),\quad\E \left(Y_1^k\bar{Y}^l\right)=O(n^{-l-1}).
\end{align*}
\end{lemma}

\begin{proof}
By Proposition \ref{prop:general_power}, $\E \left(Y_1^{k_1}\cdots Y_r^{k_r}\right)=O(n^{-r})$ for any $r\geq 1$. Recall the definition of $\bar{Y}$ in \eqref{form:def_Y_bar}, we have 
\begin{align*}
	\E\bar{Y}^l=\frac{1}{n^l}\sum_{i_1,\cdots,i_l}\E \left(Y_{i_1}\cdots Y_{i_l}\right)=O(n^{-l}).
\end{align*}
And similarly,
\begin{align*}
	\E \left(Y_1^k\bar{Y}^l\right)=\frac{1}{n^l}\sum_{i_1,\cdots,i_l}\E \left(Y_1^kY_{i_1}\cdots Y_{i_l}\right)=O(n^{-l-1}).
\end{align*}
\end{proof}

\begin{proposition}\label{prop:even_power_for_tY}
Assume that $X_1,X_2,\cdots$ are i.i.d. regularly varying with $\alpha>1$ and $\E X_1=0$. For $k\geq2$ and any $\epsilon>0$,
\begin{align*}
	\E\tY_1^{2k}=(1+o(1))\left(\E Y_1^{2k}+R_n\right),
\end{align*}
where $\widetilde{Y}_1$ is defined in \eqref{form:def_Y_tilde} and
\begin{align*}
	R_n=\begin{cases}
		O(n^{-2}), &1<\alpha\leq 2,\\
		o(n^{-\alpha+\epsilon}), &2<\alpha\leq 3,\\
		O(n^{-3}), &\alpha>3.
	\end{cases}
\end{align*}
\end{proposition}
\begin{proof}
On the one hand,
\begin{align*}
	\tY_1^{2k}=\frac{(Y_1-\bar{Y})^{2k}}{(1-n\bar{Y}^2)^k}\geq(Y_1-\bar{Y})^{2k}.
\end{align*}
On the other hand, split $\widetilde{Y}_1^{2k}$ as
\begin{align*}
	\tY_1^{2k}=\tY_1^{2k}\one(n\bar{Y}^2\leq\frac{1}{\sqrt{n}})+\tY_1^{2k}\one(n\bar{Y}^2>\frac{1}{\sqrt{n}}).
\end{align*}
By Holder's inequality and Lemma \ref{lem:bounds_for_Y_bar},
\begin{align*}
	\E\tY_1^{2k}\one(n\bar{Y}^2>\frac{1}{\sqrt{n}})\leq\sqrt{\E\tY_1^{4k}}\sqrt{P(n\bar{Y}^2>\frac{1}{\sqrt{n}})}=o(n^{-l}),
\end{align*}
for any $l>0$. So we have 
\begin{align*}
	\E\tY_1^{2k}\leq&\E\tY_1^{2k}\one(n\bar{Y}^2\leq\frac{1}{\sqrt{n}})+o(n^{-l})\leq\frac{1}{(1-\frac{1}{\sqrt{n}})^k}\E(Y_1-\bar{Y})^{2k}+o(n^{-l}).
\end{align*}
Therefore, we only need to consider 
\begin{align*}
	\E(Y_1-\bar{Y})^{2k}=\E Y_1^{2k}-2k\E \left(Y_1^{2k-1}\bar{Y}\right)+\sum_{l=0}^{2k-2}\binom{2k}{l}\E \left(Y_1^{l}\bar{Y}^{2k-l}\right).
\end{align*}
By Lemma \ref{lem:bounds_for_Y_bar}, we have 
\begin{align*}
	\sum_{l=0}^{2k-2}\binom{2k}{l}\E Y_1^l\bar{Y}^{2k-l}=O(n^{-3}).
\end{align*}
Then, we deal with  
\begin{align*}
	\E \left(Y_1^{2k-1}\bar{Y}\right)=\frac{1}{n}\sum_{i=1}^{n}\E \left(Y_1^{2k-1}Y_i\right)=\frac{1}{n}\E Y_1^{2k}+\frac{n-1}{n}\E \left(Y_1^{2k-1}Y_2\right).
\end{align*}
By Proposition \ref{prop:general_power}, we have 
\begin{align*}
	\E \left(Y_1^{2k-1}Y_2\right)=\begin{cases}
		O(n^{-2}), &1<\alpha\leq 2,\\
		o(n^{-\alpha+\epsilon}), & 2<\alpha\leq 3,\\
		O(n^{-3}), &\alpha>3,
	\end{cases}
\end{align*}
which completes the proof.

\end{proof}

\begin{proposition}\label{prop:general_power_for_tY}
Assume that $X_1,X_2,\cdots$ are i.i.d. regularly varying with $\alpha>0$. For positive integers $k_1,\cdots,k_r$, 
\begin{align*}
	\E \left(\tY_1^{k_1}\cdots \tY_r^{k_r}\right)=O(n^{-r}).
\end{align*}
\end{proposition}
\begin{proof}
When $r=1$, it is obvious that
\begin{align*}
	\E\tY_1=0,
\end{align*}
and for $k_1\geq 2$,
\begin{align*}
	|\tY_1|^{k_1}\leq \tY_1^2,
\end{align*}
which concludes that
\begin{align*}
	\E \tY_1^{k_1}=O(\frac{1}{n}).
\end{align*}
We next consider the case when $k_1,\cdots,k_r\geq 2$. Since 
\begin{align*}
	\E\left(\left|\tY_1^{k_1}\cdots\tY_r^{k_r}\right|\sum_{i=1}^n\left|\tY_i\right|^{k_{r+1}}\right)\leq\E\left|\tY_1^{k_1}\cdots\tY_r^{k_{r}}\right|,
\end{align*}
by induction on $r\geq 2$, we have
\begin{align*}
	\E\left|\tY_1^{k_1}\cdots\tY_r^{k_r}\tY_{r+1}^{k_{r+1}}\right|=O(\frac{1}{n^{r+1}}).
\end{align*}
Finally, we deal with the case when at least one index $k_i$ is $1$. We still use the induction. Suppose that
\begin{align*}
	\E \left(\tY_1^{k_1}\cdots \tY_r^{k_r}\tY_{r+1}\cdots \tY_{r+t}\right)=O(\frac{1}{n^{r+t}}),
\end{align*}
we have
\begin{align*}
	0=&\E \left(\tY_1^{k_1}\cdots \tY_r^{k_r}\tY_{r+1}\cdots \tY_{r+t}\sum_{i=1}^n \tY_i\right)\\
	=&\sum_{i=1}^r\E \left(\tY_1^{k_1}\cdots \tY_i^{k_i+1}\cdots \tY_r^{k_r}\tY_{r+1}\cdots \tY_{r+t}\right)+\sum_{j=1}^{t}\E \left(\tY_1^{k_1}\cdots \tY_r^{k_r}\tY_{r+1}\cdots \tY_{r+j}^2\cdots \tY_{r+t}\right)\\
	&+(n-r-t)\E \left(\tY_1^{k_1}\cdots \tY_r^{k_r}\tY_{r+1}\cdots \tY_{r+t}\tY_{r+t+1}\right),
\end{align*}
which implies 
\begin{align*}
	\E \left(\tY_1^{k_1}\cdots \tY_r^{k_r}\tY_{r+1}\cdots \tY_{r+t}\tY_{r+t+1}\right)=O(\frac{1}{n^{r+t+1}}).
\end{align*}
\end{proof}
\begin{proposition}\label{prop:4th_power_for_tY}
Assume that $X_1,X_2,\cdots$ are i.i.d. regularly varying with $\alpha>0$. Then
\begin{align*}
	\lim_{n\to\infty}\frac{\E\tY_1^4}{\E Y_1^4}=1.
\end{align*}
\end{proposition}
\begin{proof}
When $\alpha>1$, we have shown that 
\begin{align*}
	\E\tY_1^4=(1+o(1))(\E Y_1^4+R_n)
\end{align*}
and $R_n/\E Y_1^4\to 0$ by Proposition \ref{prop:even_power_for_tY} and Proposition \ref{prop:even_power}. So $\E\tY_1^4/\E Y_1^4\to 1$. When $\alpha\leq 1$, by Proposition \ref{prop:even_power} and \ref{form:identity_for_one}, we have 
\begin{align*}
	\E Y_1^4\sim \frac{2-\alpha}{2n},\quad\E Y_1^2Y_2^2\sim\frac{\alpha}{2n^2}.
\end{align*}
By Lemma \ref{lem:extreme_value_property}, for sufficiently small $\delta>0$,
\begin{align*}
	\E n^{1-\delta}\bar{Y}\to 0. 
\end{align*}
And by Cauchy's inequality,
\begin{align*}
	\left|\bar{Y}\right|\leq\frac{1}{\sqrt{n}}.
\end{align*}
It follows that 
\begin{align*}
	\E|\bar{Y}|^k=\E\left(|\bar{Y}|\cdot|\bar{Y}|^{k-1}\right)\leq n^{-\frac{k-1}{2}}\E|\bar{Y}|=o(n^{-\frac{k+1}{2}+\delta}),
\end{align*}
and that
\begin{align*}
	\E \left(Y_1Y_2\right)=\frac{\E n^2\bar{Y}^2-1}{n(n-1)}=o(n^{-\frac{3}{2}+\delta}).
\end{align*}
By simple algebra,
\begin{align*}
	\E\tY_1^{4}\geq\E\left(Y_1-\bar{Y}\right)^4=\E Y_1^4-4\E \left(Y_1^3\bar{Y}\right)+6\E \left(Y_1^2\bar{Y}^2\right)-4\E \left(Y_1\bar{Y}^3\right)+\E\bar{Y}^4.
\end{align*}
Note that via Proposition
\begin{align*}
	\E \left(Y_1^3\bar{Y}\right)=&\frac{1}{n}\E Y_1^4+\frac{n-1}{n}\E \left(Y_1^3Y_2\right)\leq\frac{1}{n}\E Y_1^4+\sqrt{\E Y_1^4}\sqrt{\E \left(Y_1^2Y_2^2\right)}=O(n^{-\frac{3}{2}}),\\
	\E \left(Y_1^2\bar{Y}^2\right)\leq&\sqrt{\E Y_1^4}\sqrt{\E\bar{Y}^4}=O(n^{-\frac{3}{2}}),\\
	\E \left(Y_1\bar{Y}^3\right)\leq&\sqrt{\E Y_1^2}\sqrt{\E\bar{Y}^6}=o(n^{-\frac{9}{4}+\delta}),\\
	\E\bar{Y}^4=&O(n^{-\frac{5}{2}+\delta}).
\end{align*}
We conclude that 
\begin{align*}
	\E\widetilde{Y}_1^4\geq\E Y_1^4+o(\frac{1}{n})
\end{align*}
Moreover, 
\begin{align*}
	\E\left(\tY_1^2\tY_2^2\right)\geq &\E(Y_1-\bar{Y})^2(Y_2-\bar{Y})^2\\
	=&\E \left(Y_1^2Y_2^2\right)-4\E \left(Y_1^2Y_2\bar{Y}\right)+2\E \left(Y_1^2\bar{Y}^2\right)+4\E \left(Y_1Y_2\bar{Y}^2\right)-4\E \left(Y_1\bar{Y}^3\right)+\E\bar{Y}^4.
\end{align*}
Note that
\begin{align*}
	\E \left(Y_1^2Y_2\bar{Y}\right)=&\frac{1}{n}\E \left(Y_1^2Y_2^2\right)+\frac{1}{n}\E \left(Y_1^3Y_2\right)+\frac{n-2}{n}\E \left(Y_1^2Y_2Y_3\right)\\
	=&\frac{1}{n}\E \left(Y_1^2Y_2^2\right)+\frac{1}{n}\E \left(Y_1^3Y_2\right)+\frac{1}{n}\left(\E \left(Y_1Y_2\right)-2\E \left(Y_1^3Y_2\right)\right)=o(n^{-\frac{5}{2}+\delta}),\\
	\E \left(Y_1^2\bar{Y}^2\right)=&\frac{1}{n^2}\E Y_1^4+\frac{n-1}{n^2}\E \left(Y_1^3Y_2\right)+\frac{n-1}{n^2}\E \left(Y_1^2Y_2^2\right)+\frac{(n-1)(n-2)}{n^2}\E \left(Y_1^2Y_2Y_3\right)\\
	=&o(n^{-\frac{5}{2}+\delta})\\
	\E \left(Y_1Y_2\bar{Y}^2\right)\leq&\sqrt{\E \left(Y_1^2Y_2^2\right)}\sqrt{\E\bar{Y}^4}=o(n^{-\frac{9}{4}+\delta}),\\
	\E \left(Y_1\bar{Y}^3\right)\leq&\sqrt{\E Y_1^2}\sqrt{\E\bar{Y}^6}=o(n^{-\frac{9}{4}+\delta}),\\
	\E\bar{Y}^4=&O(n^{-\frac{5}{2}+\delta}).
\end{align*}
We conclude that
\begin{align*}
	\E\tY_1^2\tY_2^2\geq\E Y_1^2Y_2^2+o(\frac{1}{n^2}).
\end{align*}
That is to say,
\begin{align*}
	\frac{\E\tY_1^4}{\E Y_1^4}\geq 1+o(1),\quad\frac{\E\tY_1^2\tY_2^2}{\E Y_1^2Y_2^2}\geq 1+o(1).
\end{align*}
If 
\begin{align*}
	\limsup_{n\to\infty}\frac{\E\tY_1^4}{\E Y_1^4}=b>1,
\end{align*}
then for sufficiently small $\varepsilon>0$, there exists $n$ such that
\begin{align*}
	1=n\E\tY_1^4+n(n-1)\E\left(\tY_1^2\tY_2^2\right)\geq& (b-\varepsilon)n\E Y_1^4+n(n-1)(1-\varepsilon)\E \left(Y_1^2Y_2^2\right)\\
	=&(b-1)n\E Y_1^4+(1-\varepsilon)>1,
\end{align*}
a contradiction.
\end{proof}

Write $\tbY=(\tY_1,\cdots,\tY_n)\trans$ where $\tY_1,\cdots,\tY_n$ are identically distributed random variables and exchangeable. Assuming $\tbY\trans \tbY=1$ and $\one_n\trans \tbY=0$, we have
\begin{gather*}
\E \tY_1=0,\quad \E \tY_1^2=\frac{1}{n},\quad \E \tY_1\tY_2=-\frac{1}{n-1}\E \tY_1^2=-\frac{1}{n(n-1)},\\
\E \tY_1^2\tY_2^2=\frac{1}{n(n-1)}-\frac{1}{n-1}\E \tY_1^4, \quad \E \tY_1^3\tY_2=-\frac{1}{n-1}\E \tY_1^4.
\end{gather*}

\begin{proposition} \label{quadratic_form_general}
For any deterministic $\A=(a_{ij})_{n\times n}$ with $\A\trans=\A, \tr(\A)=0, \A \one=\bf{0}$, we have
\begin{align*}
	\E \left(\tbY\trans \A \tbY\right)^2=\var\left(\tbY\trans \A \tbY\right)=c_1 \tr(\A\circ\A)+2 c_2 \tr(\A^2),
\end{align*}
where 
\begin{gather*}
	c_1=\frac{n^2(n+1)}{(n-1)(n-2)(n-3)}\E \tY_1^4-\frac{3}{(n-2)(n-3)},\\
	c_2=\frac{n^2-3n+3}{n(n-1)(n-2)(n-3)}-\frac{n}{(n-2)(n-3)}\E \tY_1^4.
\end{gather*}
\end{proposition}

\begin{proof}
Observe
\begin{align*}
	\E \left(\tbY\trans \A \tbY\right)=&\E \tr(\A \tbY\tbY\trans)=\tr\left( \A \left(\frac{1}{n-1}\bI_n-\frac{1}{n(n-1)} \one_n  \one_n \trans\right)\right)=0.
\end{align*}
Write
\begin{align*}
	&\var\left(\tbY\trans \A \tbY\right)=\E \left(\tbY\trans \A \tbY\right)^2= \sum_{i_1,i_2,i_3,i_4=1}^n a_{i_1 i_2} a_{i_3 i_4} \E \left(\tY_{i_1}\tY_{i_2}\tY_{i_3}\tY_{i_4}\right)\\
	=&\left\{   \sum_{|\{i_1\}\cup\cdots\cup\{i_4\}|=1}^n+\sum_{|\{i_1\}\cup\cdots\cup\{i_4\}|=2}^n+\sum_{|\{i_1\}\cup\cdots\cup\{i_4\}|=3}^n+\sum_{|\{i_1\}\cdots\cup\{i_4\}|=4}^n\right\}  a_{i_1 i_2} a_{i_3 i_4} \E \left(\tY_{i_1}\tY_{i_2}\tY_{i_3}\tY_{i_4}\right)\\
	=& \E \tY_1^4 \sum_{i=1}^n a^2_{ii}+\E \left(\tY_1^2 \tY_2^2\right) \sum_{i\neq j} \left(a_{ii} a_{jj}+2a_{ij} a_{ji}\right)+4\E \left(\tY_1^3 \tY_2\right) \sum_{i \neq j} a_{ii}a_{ij} \\
	&+\E \left(\tY_1^2 \tY_2 \tY_3\right) \sum_{i\neq j \neq k}\left(2a_{ij} a_{kk}+4a_{ik}a_{jk}\right)+ \E \left(\tY_1 \tY_2 \tY_3 \tY_4\right) \sum_{i\neq j \neq k \neq l} a_{ij} a_{kl}\\
	=& \left(\E \tY_1^4-\E \left(\tY_1 \tY_2 \tY_3 \tY_4\right)\right) \sum_{i=1}^n a^2_{ii}+\left(\E \left(\tY_1^2 \tY_2^2\right)-\E \left(\tY_1 \tY_2 \tY_3 \tY_4\right)\right) \sum_{i\neq j} \left(a_{ii} a_{jj}+2a_{ij} a_{ji}\right)\\&+4\left(\E \left(\tY_1^3 \tY_2\right)-\E \left(\tY_1 \tY_2 \tY_3 \tY_4\right)\right) \sum_{i \neq j} a_{ii}a_{ij}\\ 
	&+\left(\E \left(\tY_1^2 \tY_2 \tY_3\right)-\E \left(\tY_1 \tY_2 \tY_3 \tY_4\right)\right) \sum_{i\neq j \neq k}\left(2a_{ij} a_{kk}+4a_{ik}a_{jk}\right)\\
	=& \left(\E \tY_1^4-3\E \left(\tY_1^2 \tY_2^2\right)+2\E \left(\tY_1 \tY_2 \tY_3 \tY_4\right)\right) \tr(\A\circ\A)+2 \left(\E \left(\tY_1^2 \tY_2^2\right)-\E \left(\tY_1 \tY_2 \tY_3 \tY_4\right)\right) \tr(\A^2)\\&+4\left(\E \left(\tY_1^3 \tY_2\right)-3\E \left(\tY_1^2 \tY_2 \tY_3\right)+2\E \left(\tY_1 \tY_2 \tY_3 \tY_4\right)\right) \sum_{i \neq j} a_{ii}a_{ij}\\ 
	&+4\left(\E \left(\tY_1^2 \tY_2 \tY_3\right)-\E \left(\tY_1 \tY_2 \tY_3 \tY_4\right)\right) \left(\one \trans \A^2 \one -\tr(\A^2)\right)\\
	=& \left(\E \tY_1^4-3\E \left(\tY_1^2 \tY_2^2\right)-4\E \left(\tY_1^3 \tY_2\right)+12\E \left(\tY_1^2 \tY_2 \tY_3\right)-6\E \left(\tY_1 \tY_2 \tY_3 \tY_4\right)\right) \tr(\A\circ\A)\\
	&+2 \left(\E \left(\tY_1^2 \tY_2^2\right)-2\E \left(\tY_1^2 \tY_2 \tY_3\right)+\E \left(\tY_1 \tY_2 \tY_3 \tY_4\right)\right) \tr(\A^2),
\end{align*}
where we use the facts that
\begin{align*}
	&0=\left(\one \trans \A \one \right)^2=\sum_{i_1,i_2,i_3,i_4=1}^n a_{i_1 i_2} a_{i_3 i_4}\\
	=&\sum_{i=1}^n a^2_{ii}+\sum_{i\neq j} \left(a_{ii} a_{jj}+2a_{ij} a_{ji}\right)+4 \sum_{i \neq j} a_{ii}a_{ij}\\
	& +\sum_{i\neq j \neq k}\left(2a_{ij} a_{kk}+4a_{ik}a_{jk}\right) +  \sum_{i\neq j \neq k \neq l} a_{ij} a_{kl},\\
	&0=\left(\one \trans \A \one \right) \tr(\A)=\sum_{i_1,i_2,i_3=1}^n a_{i_1 i_2} a_{i_3 i_3}\\
	=& \sum_{i\neq j \neq k} a_{ij} a_{kk}+\sum_{i \neq j}\left(2 a_{ii}a_{ij}+a_{ii}a_{jj}\right) +\sum_{i=1}^n a^2_{ii}\\
	=&\sum_{i\neq j \neq k} a_{ij} a_{kk}+2 \sum_{i \neq j} a_{ii}a_{ij}+\sum_{i,j=1}^n  a_{ii}a_{jj}=\sum_{i\neq j \neq k} a_{ij} a_{kk}+2 \sum_{i \neq j} a_{ii}a_{ij},\\
	& \sum_{i\neq j \neq k} a_{ik}a_{jk}=\sum_{i,j,k=1}^n a_{ik}a_{jk}-\sum_{i \neq j}\left(2a_{ii}a_{ij}+a^2_{ij}  \right)  - \sum_{i=1}^n a^2_{ii}\\ =&\one \trans \A^2 \one -\tr(\A^2)-2\sum_{i \neq j}a_{ii}a_{ij},\\
	&\sum_{i\neq j} \left(a_{ii} a_{jj}+2a_{ij} a_{ji}\right)=\sum_{i,j=1}^n \left(a_{ii} a_{jj}+2a_{ij} a_{ji}\right)-3\sum_{i=1}^n a^2_{ii}\\
	=&2\tr(\A^2)-3\tr(\A \circ \A),\\
	& \sum_{i \neq j}a_{ii}a_{ij}=\sum_{i,j=1}^n a_{ii}a_{ij}-\sum_{i=1}^n a^2_{ii}=\one \trans \A \diag(\A)-\tr(\A \circ \A).
\end{align*}
By exploring the exchangeability, we have
\begin{align*}
	1=&\E (\tY_1^2+\cdots+\tY_n^2)^2=n \E \tY_1^4+n(n-1)\E \tY_1^2\tY_2^2 \\
	\implies &  \E \tY_1^2\tY_2^2=\frac{1}{n(n-1)}-\frac{1}{n-1}\E \tY_1^4,\\
	0=&\E \tY_1^3 (\tY_1+\cdots+\tY_n)=\E \tY_1^4+(n-1)\E \tY^3_1\tY_2 \implies  \E \tY_1^3\tY_2=-\frac{1}{n-1}\E \tY_1^4,\\
	0=&\E \tY_1^2 \tY_2 (\tY_1+\cdots+\tY_n)=\E \tY^3_1\tY_2+\E \tY_1^2\tY_2^2+(n-2)\E \tY^2_1\tY_2\tY_3\\
	\implies & \E \tY_1^2\tY_2\tY_3=\frac{2}{(n-1)(n-2)}\E \tY_1^4-\frac{1}{n(n-1)(n-2)},\\
	0=&\E \tY_1 \tY_2\tY_3 (\tY_1+\cdots+\tY_n)=3\E \tY_1^2\tY_2\tY_3+(n-3)\E \tY_1 \tY_2\tY_3\tY_4\\
	\implies &\E \tY_1 \tY_2\tY_3\tY_4=\frac{3}{n(n-1)(n-2)(n-3)}-\frac{6}{(n-1)(n-2)(n-3)}\E \tY_1^4
\end{align*}
which yield 
\begin{align*}
	&\E \tY_1^4-3\E \tY_1^2 \tY_2^2-4\E \tY_1^3 \tY_2+12\E \tY_1^2 \tY_2 \tY_3-6\E \tY_1 \tY_2 \tY_3 \tY_4\\
	=&\frac{n^2(n+1)}{(n-1)(n-2)(n-3)}\E \tY_1^4-\frac{3}{(n-2)(n-3)},\\
	&\E \tY_1^2 \tY_2^2-2\E \tY_1^2 \tY_2 \tY_3+\E \tY_1 \tY_2 \tY_3 \tY_4\\
	=&\frac{n^2-3n+3}{n(n-1)(n-2)(n-3)}-\frac{n}{(n-2)(n-3)}\E \tY_1^4.
\end{align*}
\end{proof}

\begin{proposition} \label{all_moments}
Letting $\tbY_1, \tbY_2, \tbY_3$ be independent copies of $\tbY$, we have
\begin{align*}
	\E(\tbY_1\trans \tbY_2)^2=&\frac{1}{n-1}, \quad  \cov\left((\tbY_1\trans \tbY_2)^2,(\tbY_1\trans \tbY_3)^2\right)=0,\\
	\var\left((\tbY_1\trans \tbY_2)^2\right)=&c_1 \left(n\E \tY_1^4-\frac{1}{n}\right)+2 c_2 \frac{n-2}{n-1}\\
	=&\left(n \left(\E \tY_1^4\right)^2+\frac{2}{n^2}\right)(1+o(1)).
\end{align*}
\end{proposition}

\begin{proof}
Noting 
\begin{align*}
	\E \tbY\tbY\trans=(\E\tY_1^2-\E \tY_1\tY_2)\bI_n+\E \tY_1\tY_2 \one_n  \one_n \trans=\frac{1}{n-1}\bI_n-\frac{1}{n(n-1)} \one_n  \one_n \trans,
\end{align*}
then we can get 
\begin{align*}
	\E(\tbY_1\trans \tbY_2)^2=&\E \tr(\tbY_1\tbY_1\trans\tbY_2\tbY_2\trans)=\tr\left( \frac{1}{n-1}\bI_n-\frac{1}{n(n-1)} \one_n  \one_n \trans \right)^2=\frac{1}{n-1}. 
\end{align*}   
and  
\begin{align*}
	\cov\left((\tbY_1\trans \tbY_2)^2,(\tbY_1\trans \tbY_3)^2\right)=\var\left( \tbY_1\trans\left(\frac{1}{n-1}\bI_n-\frac{1}{n(n-1)} \one_n  \one_n \trans \right)\tbY_1 \right)=0.
\end{align*}
For the covariance, noting
\begin{align*}
	&\left(\tbY_1 \trans\tbY_2 \right)^2=\tbY_1\trans \left( \tbY_2\tbY_2\trans\right) \tbY_1  \\
	=&\tbY_1\trans \left( \tbY_2\tbY_2\trans-\frac{1}{n-1}\bI_n+\frac{1}{n(n-1)} \one_n  \one_n \trans \right) \tbY_1  +\frac{1}{n-1},
\end{align*}
we can get
\begin{align*}
	\var\left( \left(\tbY_1 \trans\tbY_2 \right)^2 \right)=\E \left(\tbY_1\trans \left( \tbY_2\tbY_2\trans-\frac{1}{n-1}\bI_n+\frac{1}{n(n-1)} \one_n  \one_n \trans\right) \tbY_1  \right)^2.
\end{align*}
Writing 
\begin{align*}
	\A=\tbY_2\tbY_2\trans-\frac{1}{n-1}\bI_n+\frac{1}{n(n-1)} \one_n  \one_n\trans,
\end{align*}
we have $\A\trans=\A, \tr(\A)=0, \A \one=\bf{0}$ and 
\begin{align*}
	\tr(\A\circ\A)=\sum_{k=1}^n\left(\tY_{2k}^2-\frac{1}{n}\right)^2=\sum_{k=1}^n \tY_{2k}^4-\frac{1}{n},\quad \tr(\A^2)=\frac{n-2}{n-1}.
\end{align*}
By Proposition \ref{quadratic_form_general}, we can conclude
\begin{align*}
	&\var\left( \left(\tbY_1 \trans\tbY_2 \right)^2 \right)= \E \left(\tbY_1\trans \A \tbY_1\right)^2=c_1 \E \tr(\A\circ\A)+2 c_2 \E \tr(\A^2)\\
	=& c_1 \left(n\E \tY_1^4-\frac{1}{n}\right)+2 c_2 \frac{n-2}{n-1}\\
	=&\frac{n^3(n+1)}{(n-1)(n-2)(n-3)}\left(\E \tY_1^4\right)^2-\frac{6n}{(n-2)(n-3)}\E \tY_1^4\\
	&+\frac{2n^3-7n^2+12n-9}{n(n-1)^2(n-2)(n-3)}\\
	=&\left(n \left(\E \tY_1^4\right)^2+\frac{2}{n^2}\right)(1+o(1)).
\end{align*}
\end{proof}

\begin{proposition} \label{all_moments_Hadamard}
Letting $\tbY_1, \tbY_2, \tbY_3$ be independent copies of $\tbY$, we have
\begin{gather*}
	\E \left(\left\|\tbY\circ \tbY\right\|^2_2\right)=n \E \tY_1^4,\\
	\var \left(\left\|\tbY\circ \tbY\right\|^2_2\right)=n\E \tY_1^8+n(n-1)\E \tY_1^4\tY_2^4-n^2 \left(\E \tY_1^4\right)^2,
\end{gather*}
and
\begin{gather*}
	\E \left(\left\|\tbY_1 \circ \tbY_2\right\|^2_2\right)=\frac{1}{n},\quad  \cov\left(\left\|\tbY_1 \circ \tbY_2\right\|^2_2,\left\|\tbY_1 \circ \tbY_3\right\|^2_2\right)=0,\\
	\var\left( \left\|\tbY_1 \circ \tbY_2\right\|^2_2 \right)=\frac{\left(n^2 \E \tY_1^4-1 \right)^2}{n^2(n-1)}.
\end{gather*}

\end{proposition}
\begin{proof}
By the definition of Hadamard product,
\begin{align*}
	\left\|\tbY\circ \tbY\right\|^2_2=\sum_{i=1}^n \tY_i^4,   
\end{align*}
which yields 
\begin{align*}
	\E \left(\left\|\tbY\circ \tbY\right\|^2_2\right)=n \E \tY_1^4,  
\end{align*}
and
\begin{align*}
	\var \left(\left\|\tbY\circ \tbY\right\|^2_2\right)=&\sum_{i,j=1}^n \E \tY_i^4 \tY_j^4- n^2 \left(\E \tY_1^4\right)^2,\\
	=&n\E \tY_1^8+n(n-1)\E \tY_1^4\tY_2^4-n^2 \left(\E \tY_1^4\right)^2
\end{align*}
For the Hadamard product,
\begin{align*}
	&\left\|\tbY_1 \circ \tbY_2\right\|^2_2-\frac{1}{n}\\
	&=\tr\left(\left(\tbY_1\tbY_1\trans-\frac{1}{n-1}\left(\bI_n-\frac{1}{n} \one_n  \one_n\trans\right) \right) \circ \left(\tbY_2\tbY_2\trans-\frac{1}{n-1}\left(\bI_n-\frac{1}{n} \one_n  \one_n\trans\right) \right)    \right), 
\end{align*}
which yields
\begin{align*}
	\E \left(\left\|\tbY_1 \circ \tbY_2\right\|^2_2\right)=\frac{1}{n},\quad  \cov\left(\left\|\tbY_1 \circ \tbY_2\right\|^2_2,\left\|\tbY_1 \circ \tbY_3\right\|^2_2\right)=0.
\end{align*}
For the covariance, 
\begin{align*}
	\var\left(\left\|\tbY_1 \circ \tbY_2\right\|^2_2\right)=&\E \left(\sum_{i=1}^n \tY_{1i}^2\tY_{2i}^2\right)^2-\frac{1}{n^2}=\sum_{i,j=1}^n \left(\E \tY_{1i}^2 \tY_{1j}^2\right)^2-\frac{1}{n^2},\\
	=&n \left(\E \tY_1^4\right)^2+n(n-1)\left(\E \tY_1^2\tY_2^2\right)^2-\frac{1}{n^2}\\
	=&\frac{n^2}{n-1}\left(\E \tY_1^4\right)^2-\frac{2}{n-1}\E \tY_1^4+\frac{1}{n^2(n-1)}\\
	=&\frac{\left(n^2 \E \tY_1^4-1 \right)^2}{n^2(n-1)}.
\end{align*}
\end{proof}


\begin{proof}[Proof of Theorem \ref{thm:CLT3}]  
	Let $\tV_n=\var(\tT)$. Take the martingale decomposition of $\tT$ as follows:
	\begin{align*}
		&\tM_{n,k}=\frac{1}{\sqrt{\tV_n}}\left(\E[\tT|\tbY_1,\cdots,\tbY_k]-\E[\tT|\tbY_1,\cdots,\tbY_{k-1}]\right)\\
		=&\frac{2}{\sqrt{\tV_n}}\tr \left(  \sum_{i<k} \tbY_i\tbY_i\trans+\frac{p-k}{n-1} \bI_n-\frac{p-k}{n(n-1)}\one_n  \one_n \trans\right)\\
		&\left(  \tbY_k \tbY_k\trans-\frac{1}{n-1}\bI_n+\frac{1}{n(n-1)}\one_n  \one_n\trans\right)\\
		=&\frac{2}{\sqrt{\tV_n}}\tr \left(  \sum_{i<k} \tbY_i\tbY_i\trans \right)\left(  \tbY_k \tbY_k\trans-\frac{1}{n-1}\bI_n+\frac{1}{n(n-1)}\one_n  \one_n \trans \right)\\
		=&\frac{2}{\sqrt{\tV_n}} \sum_{i<k} \tbY_k \trans  \left(\tbY_i \tbY_i \trans-\frac{1}{n-1}\bI_n+\frac{1}{n(n-1)}\one_n  \one_n \trans\right)\tbY_k,
	\end{align*}
	where we use the fact
	\begin{align*}
		\tT=&1+\tr \left(  \sum_{i \neq k} \tbY_i\tbY_i\trans \right)\left(  \sum_{j \neq k} \tbY_j\tbY_j\trans \right)\\
		&+2 \tr \left(  \sum_{i<k} \tbY_i\tbY_i\trans+\sum_{i>k}\tbY_i\tbY_i\trans \right)\left(  \tbY_k \tbY_k\trans\right).
	\end{align*}
	The remaining details of the proof are very similar to those in the symmetric distribution case discussed earlier, and we only outline the key steps here.
	\bigskip

	Write
	\begin{align*}
		\A_i=\tbY_i \tbY_i \trans-\frac{1}{n-1}\bI_n+\frac{1}{n(n-1)}\one_n  \one_n \trans.
	\end{align*}
	One can verify that
	\begin{align*}
		\A_i\trans=\A_i, \quad \tr(\A_i)=0, \quad \A_i \one =\bf{0}.
	\end{align*}
	By Proposition \ref{quadratic_form_general}
	\begin{align*}
		&\E[\tM_{n,k}^2|\tbY_1,\cdots,\tbY_{k-1}]\\
		=&\frac{4c_1}{\tV_n}\tr\left(\left(\sum_{i=1}^{k-1}\A_i \right)\circ \left(\sum_{i=1}^{k-1}\A_i \right) \right)+\frac{8c_2}{\tV_n}\tr\left(\left(\sum_{i=1}^{k-1}\A_i \right)^2\right)\\
		=&\frac{4c_{1}}{\tV_n}\sum_{i,j=1}^{k-1} \left(\left\|\tbY_i \circ \tbY_j\right\|^2_2-\frac{1}{n}\right)\\
		&+\frac{16 c_{2}}{\tV_n}\sum_{i<j<k} \left(\left( \tbY_i \trans \tbY_j\right)^2-\frac{1}{n-1}\right)+\frac{8(k-1)c_{2}}{\tV_n}\frac{n-2}{n-1}.
	\end{align*}
	By calculating the variances of 
	\begin{gather*}
		\left\|\tbY_i \circ \tbY_i\right\|^2_2,\quad \left\|\tbY_1 \circ \tbY_1\right\|^2_2,\quad \left( \tbY_i \trans \tbY_j\right)^2,
	\end{gather*}
	we can conclude
	\begin{align} \label{condition3}
		\sum_{k=1}^p\E[M_{n,k}^2|\tbY_1,\cdots,\tbY_{k-1}] \leq  \frac{C}{p} \frac{\left(\E \tY_1^4\right)^2 \left(n\E \tY_1^8+n(n-1)\E \tY_1^4\tY_2^4\right) }{\left(n \left(\E \tY_1^4\right)^2+\frac{2}{n^2}\right)^2}+o(1).
	\end{align}
	
	Noting
	\begin{align*}
		\tM_{n,k}=& \frac{2}{\sqrt{\tV_n}} \sum_{i=1}^{k-1} \tbY_i \trans  \left(  \tbY_k \tbY_k\trans-\frac{1}{n-1}\bI_n+\frac{1}{n(n-1)}\one_n  \one_n \trans \right) \tbY_i\\
		=&\frac{2}{\sqrt{\tV_n}} \sum_{i=1}^{k-1} \tbY_i \trans  \A_k \tbY_i
	\end{align*}
	we have
	\begin{align*}
		\tM_{n,k}^4=&\frac{16}{\tV_n^2} \sum_{i_1,i_2,i_3,i_4=1}^{k-1} \left(\tbY_{i_1} \trans \A_k \tbY_{i_1}\right) \left(\tbY_{i_2} \trans \A_k \tbY_{i_2}\right)  \left(\tbY_{i_3} \trans \A_k \tbY_{i_3}\right) \left( \tbY_{i_4} \trans \A_k \tbY_{i_4}\right).
	\end{align*}
	Since
	\begin{align*}
		\E \left( \tbY_i \trans  \A_k \tbY_i \mid \tbY_k \right)=\tr\left( \frac{1}{n-1}\bI_n-\frac{1}{n(n-1)}\one_n  \one_n \trans \right) \A_k=0,~i=1,\cdots,k-1,
	\end{align*}
	we can get 
	\begin{align*}
		&\E \tM_{n,k}^4= \frac{16}{\tV_n^2} \sum_{i=1}^{k-1} \E \left(\tbY_{i} \trans \A_k \tbY_{i}\right)^4+ \frac{48}{\tV_n^2} \sum_{i \neq j}^{k-1} \E \left(\tbY_{i} \trans \A_k \tbY_{i}\right)^2\left(\tbY_{j} \trans \A_k \tbY_{j}\right)^2\\
		=& \frac{16(k-1)}{\tV_n^2} \E \left(\tbY_{1} \trans \A_2 \tbY_{1}\right)^4+ \frac{48(k-1)(k-2)}{\tV_n^2} \E \left(\tbY_{1} \trans \A_3 \tbY_{1}\right)^2\left(\tbY_{2} \trans \A_3 \tbY_{2}\right)^2.
	\end{align*}
	For the first term, we have
	\begin{align*}  
		\E \left(\tbY_{1} \trans \A_2 \tbY_{1}\right)^4=\E \left( \left(\tbY_1\trans \tbY_2 \right)^2-\frac{1}{n-1}  \right)^4\leq 8 \E \left(\left(\tbY_1\trans \tbY_2 \right)^8+\frac{1}{n^4}\right).
	\end{align*}
	Then
	\begin{align*}
		&\E\left(\tbY_1\trans\tbY_2\right)^8=\E \sum_{i_1,\cdots,i_8=1}^n \tY_{1i_1}\cdots \tY_{1i_8}\tY_{2i_1}\cdots \tY_{2i_8}\\
		= &\left(\sum_{|\{i_1\}\cup\cdots \cup \{i_8\}|=1} + \sum_{|\{i_1\}\cup\cdots \cup \{i_8\}|=2}
		+\sum_{|\{i_1\}\cup\cdots \cup \{i_8\}|=3}+\sum_{|\{i_1\}\cup\cdots \cup \{i_8\}|\geq 4}\right)\\
		& \left(\E \tY_{i_1}\cdots \tY_{i_8}\right)^2\\
		=&n\left(\E \tY_{1}^8\right)^2+n(n-1) \sum_{k_1+k_2=8,k_1,k_2\geq 1} \frac{8!}{k_1!k_2!} \left(\E \tY_{1}^{k_1}\tY_{2}^{k_2}\right)^2\\
		&+n(n-1)(n-2) \sum_{k_1+k_2+k_3=8,k_1,k_2,k_3\geq 1} \frac{8!}{k_1!k_2!k_3!}\left(\E \tY_{1}^{k_1}\tY_{2}^{k_2}\tY_{3}^{k_3}\right)^2+O(n^{-4}),
	\end{align*} 
	where the last step is due to the fact that $\E \tY_1^{k_1}\cdots \tY_r^{k_r}=O(n^{-r})$ by Proposition \ref{prop:general_power_for_tY}. For the moments with odd powers, we can bound them by the moments with even powers. For example,
	\begin{align*}
		& \E \tY_1^7 \tY_2=-\frac{1}{n-1} \E \tY_1^8, \\
		& \E \tY_1^5\tY_2^3\leq \frac{1}{2} \E \tY_1^4\tY_2^2 \left(\tY_1^2+\tY_2^2\right)\leq  \frac{1}{2}\E \tY_1^4\tY_2^4 +\frac{1}{2}\E \tY_1^6\tY_2^2.
	\end{align*}
	For $k_1+k_2+k_3=8$, we have
	\begin{align*}
		0= &\E \tY_1^k \tY^{7-k}_2 \left(\tY_1+\cdots+\tY_n\right)\\
		=&\E \tY_1^{k+1}\tY_2^{7-k}+\E \tY_1^k \tY^{8-k}_2+(n-1)\E \tY_1^{k} \tY^{7-k}_2\tY_3,
	\end{align*}
	which implies
	\begin{gather*}
		\E \tY_1^{k} \tY^{7-k}_2\tY_3=-\frac{1}{n-1} \E \tY_1^{k+1}\tY_2^{7-k}-\frac{1}{n-1} \E \tY_1^{k}\tY_2^{8-k},~k=1,\cdots,6. 
	\end{gather*}
	Moreover,
	\begin{align*}
		\E \tY_1^3\tY_2^3\tY_3^2 \leq \frac{1}{2} \E \tY_1^2\tY_2^2 \tY_3^2 \left(\tY_1^2+\tY_2^2\right)\leq  \E \tY_1^4\tY_2^2\tY_3^2.
	\end{align*}   
	Combining all the pieces together, we can get
	\begin{align*}
		& \E\left(\tbY_1\trans\tbY_2\right)^8 \\
		\leq & C \left( n\left(\E \tY_{1}^8\right)^2+n^2 \left(\E \tY_{1}^6\tY_2^2\right)^2+n^2 \left(\E \tY_{1}^4 \tY_2^4\right)^2+n^3 \left(\E \tY_{1}^2 \tY_2^2 \tY_3^4 \right)^2+n^{-4}\right).
	\end{align*}
	For the second term, by Proposition \ref{quadratic_form_general}, we have
	\begin{align*}
		&  \E \left(\tbY_{1} \trans \A_3 \tbY_{1}\right)^2\left(\tbY_{2} \trans \A_3 \tbY_{2}\right)^2
		=\E \left( c_1 \tr(\A_3 \circ\A_3)+2 c_2 \tr(\A_3^2)    \right)^2\\
		=& \E \left(c_1 \left\|\tbY_1 \circ \tbY_1\right\|^2_2-\frac{c_1}{n}+2c_2\frac{n-2}{n-1}\right)^2=c_1^2 \var\left(\left\|\tbY_1 \circ \tbY_1\right\|^2_2  \right)+4c_2^2\frac{(n-2)^2}{(n-1)^2}.
	\end{align*}
	Therefore, the conditions of martingale central limit theorem hold if 
	\begin{gather*}
		\frac{n\left(\E \tY_{1}^8\right)^2+n^2 \left(\E \tY_{1}^6\tY_2^2\right)^2+n^2 \left(\E \tY_{1}^4 \tY_2^4\right)^2+n^3 \left(\E \tY_{1}^2 \tY_{2}^2 \tY_{3}^4 \right)^2}{p^2\left(n \left(\E \tY^4\right)^2+\frac{2}{n^2}\right)^2} \to 0;  \\
		\frac{\left(\E \tY^4\right)^2 \left( n\E \tY_1^8+n(n-1)\E \tY_1^4\tY_2^4\right)}{p\left(n \left(\E \tY^4\right)^2+\frac{2}{n^2}\right)^2} \to 0.
	\end{gather*}
	Therefore, it suffices to have
	\begin{align*} 
		\frac{n \E \tY_{1}^8+\E \tY_{1}^4}{\sqrt{n} p \left(n \left(\E \tY^4\right)^2+\frac{2}{n^2}\right)} \to 0, 
	\end{align*}  
	which holds under the following relationship between p and n
	\begin{align*}
		\begin{cases}
			\frac{n}{p^2}\to 0,&0<\alpha<2,\\
			\frac{n^{\alpha-1}}{p^2} \frac{\tilde{l}^2(\sqrt{a_n})}{l^2(\sqrt{a_n})}  \to 0, &\alpha=2,\E X^{2}=\infty,\\
			\frac{n^{\alpha-1}}{p^2} \frac{1}{l^2(\sqrt{n})}  \to 0,&2\leq\alpha<3, \E X^{2}<\infty\\
			\frac{n}{p}\frac{l(\sqrt{n})}{l^2(\sqrt{n})+2}\to 0,&\alpha=3,\\
			\frac{n^{5-\alpha}l^2(\sqrt{n})}{p^2}\to 0,&3<\alpha\leq 5,\\
			n, p \to \infty,&\alpha>5.
		\end{cases}
	\end{align*}
\end{proof}

\begin{proof}[Proof of Theorem \ref{thm:CLT_distribution_free}]
	Let
	\begin{align*}
		H=\frac{1}{np}\sum_{i,j}\tY_{ij}^4.
	\end{align*}
	We need to verify the convergence of our statistic
	\begin{align*}
		\frac{\tr(\widetilde{\R}_n^2)-\E\tr(\widetilde{\R}_n^2)}{\sqrt{\widehat{V}}}=\frac{\tr(\widetilde{\R}_n^2)-\E\tr(\widetilde{\R}_n^2)}{\sqrt{\frac{2p^2}{n^2}\left[n^3\left(\E \tY_{11}^4\right)^2+2\right]}}\cdot\sqrt{\frac{n^3\left(\E \tY_{11}^4\right)^2+2}{n^3H^2+2}}.
	\end{align*}
	To this end, it suffices to show that
	\begin{align*}
		\frac{n^3\left(\E \tY_{11}^4\right)^2+2}{n^3H^2+2}\stackrel{P}{\to}1.
	\end{align*}
	If $\alpha>3$, it suffices to prove that
	\begin{align*}
		n^{\frac{3}{2}}H\stackrel{P}{\to}0.
	\end{align*}
	If $\alpha\leq3$, it suffices to prove that 
	\begin{align*}
		\frac{\E \tY_1^4}{H}\stackrel{P}{\to}1.
	\end{align*}
	
	Consider the case when $\alpha\leq3$ first. By Markov's inequality  
	\begin{align*}
		P\left(\left|\frac{H}{\E \tY_{11}^4}-1\right|>\varepsilon\right)=P\left(\left|H-\E \tY_{11}^4\right|>\varepsilon\E \tY_{11}^4\right)\leq\frac{\E\left|H-\E \tY_{11}^4\right|^2}{\varepsilon^2\left(\E \tY_{11}^4\right)^2}.
	\end{align*}
	A straightforward calculation indicates that
	\begin{align*}
		&\E|H-\E \tY_{11}^4|^2=\frac{1}{n^2p^2}\sum_{i,k=1}^p\sum_{j,l=1}^n\E(\tY_{ij}^4-\E \tY_{ij}^4)(\tY_{kl}^4-\E \tY_{kl}^4)\\
		=&\frac{1}{n^2p^2}\sum_{i=1}^p\sum_{j,l=1}^n\E(\tY_{ij}^4-\E \tY_{ij}^4)(\tY_{il}^4-\E \tY_{il}^4)\\
		&+\frac{1}{n^2p^2}\sum_{i\not=k}^p\sum_{j,l=1}^n\E(\tY_{ij}^4-\E \tY_{ij}^4)\E(\tY_{kl}^4-\E \tY_{kl}^4)\\
		=&\frac{1}{n^2p}\sum_{j,l=1}^n\E(\tY_{1j}^4-\E \tY_{1j}^4)(\tY_{1l}^4-\E \tY_{1l}^4)\\
		=&\frac{1}{n^2p}\sum_{j=1}^n\E(\tY_{1j}^4-\E \tY_{1j}^4)^2+\frac{1}{n^2p}\sum_{j\not=l}^n\E(\tY_{1j}^4-\E \tY_{1j}^4)(\tY_{1l}^4-\E \tY_{1l}^4)\\
		=&\frac{1}{np}\E(\tY_{11}^4-\E \tY_{11}^4)^2+\frac{1}{p}\E(\tY_{11}^4-\E \tY_{11}^4)(\tY_{12}^4-\E \tY_{11}^4).
	\end{align*}
	From Proposition \ref{prop:even_power_for_tY} we conclude that
	\begin{align*}
		\frac{\E|H-\E \tY_{11}^4|^2}{(\E \tY_{11}^4)^2}=\frac{\E \tY_{11}^8}{(\E \tY_{11}^4)^2}O(\frac{1}{np})+\frac{\E\tY_{11}^4\tY_{12}^4}{(\E\tY_{11}^4)^2}O(\frac{1}{p})=O(\frac{\sqrt{n}}{p})\to 0.
	\end{align*}

	Consider the case when $\alpha>3$ next. We only need to show that 
	\begin{align*}
		\E n^{\frac{3}{2}}H\to 0.
	\end{align*}
	Proposition \ref{prop:even_power_for_tY} yields
	\begin{align*}
		n^{\frac{3}{2}}\E H=n^{\frac{3}{2}}\E \tY_{11}^4\to 0.
	\end{align*}
	This completes the proof.
\end{proof}

\end{document}